\documentclass[11pt, reqno]{amsproc} 

\usepackage{geometry}
\usepackage{color}
\usepackage{verbatim}
\usepackage[utf8x]{inputenc}
\usepackage{t1enc}
\usepackage[english]{babel}

\usepackage{amsmath,amssymb,amsthm}
\usepackage{mathrsfs,esint}
\usepackage{graphicx,wrapfig}
\usepackage[format=hang]{caption}
\usepackage{hyperref}

\newcommand*{\R}{{\mathbb R}}

\newcommand*{\N}{{\mathbb N}}
\newcommand*{\Z}{{\mathbb Z}}

\newcommand*{\eps}{\varepsilon}

\newcommand*{\Om}{\Omega}

\newcommand*{\pip}{\varphi}

\newcommand*{\barnu}{\overline{\nu}}

\newcommand*{\M}{\mathcal M}

\providecommand*{\vint}[1]{\mathchoice
          {\mathop{\vrule width 5pt height 3 pt depth -2.5pt
                  \kern -9pt \kern 1pt\intop}\nolimits_{\kern -5pt{#1}}}
          {\mathop{\vrule width 5pt height 3 pt depth -2.6pt
                  \kern -6pt \intop}\nolimits_{\kern -3pt{#1}}}
          {\mathop{\vrule width 5pt height 3 pt depth -2.6pt
                  \kern -6pt \intop}\nolimits_{\kern -3pt{#1}}}
          {\mathop{\vrule width 5pt height 3 pt depth -2.6pt
                  \kern -6pt \intop}\nolimits_{\kern -3pt{#1}}}}
\newcommand*{\jint}{\fint}

\DeclareMathOperator{\dist}{dist}
\DeclareMathOperator{\diam}{diam}

\DeclareMathOperator{\osc}{osc}

\numberwithin{equation}{section}
\theoremstyle{plain}
\newtheorem{definition}[equation]{Definition}
\newtheorem{thm}[equation]{Theorem}

\newtheorem{prop}[equation]{Proposition}
\newtheorem{cor}[equation]{Corollary}
\newtheorem{lem}[equation]{Lemma}

\DeclareMathOperator*{\elimsup}{ess\,lim\,sup}

\DeclareMathOperator{\rcapa}{cap}

\theoremstyle{definition}

\newtheorem{defn}[equation]{Definition}

\newtheorem{remark}[equation]{Remark}

\begin{document}

\title[Sharp H\"older exponent for the Neumann problem]
{Sharp H\"older regularity of weak solutions of the Neumann problem 
and applications to nonlocal PDE
{in  
 metric measure spaces}} 

\begin{abstract} We prove global  H\"older regularity result for weak solutions $u\in N^{1,p}(\Omega, \mu)$ to a PDE of $p$-Laplacian type with a measure as non-homogeneous term:
\[
-\text{div}\!\left( |\nabla u|^{p-2}\nabla u \right)=\overline\nu,
\]	
where $1<p<\infty$ and $\overline\nu \in (N^{1,p}(\Omega,\mu))^*$ is a signed Radon measure supported in $\overline \Omega$. Here, $\Omega$ is a John domain in a metric measure space
satisfying a doubling condition and a $p$-Poincar\'e inequality, and $\nabla u$ is the Cheeger gradient. The regularity results obtained in this paper improve on earlier estimates proved by the authors in \cite{CGKS} for the study of the Neumann problem, and have applications to the regularity of solutions of nonlocal PDE in doubling metric spaces. Moreover, the obtained H\"older exponent matches with the known sharp result in the Euclidean case \cite{CSt,BLS,BT}. 
\end{abstract}

\author{Luca Capogna}
\address{Department of Mathematical Sciences, Smith College, Northampton, MA, 01060, USA}
\email{lcapogna@smith.edu}
\author{Ryan Gibara}
\address{Department of Mathematics, Physics and Geology, Cape Breton University, Sydney, NS~B1Y3V3, Canada}
\email{ryan\_gibara@cbu.ca}
\author{Riikka Korte}
\address{Department of Mathematics and Systems Analysis, Aalto University, P.O.~Box 11100,
FI-00076 Aalto, Finland}
\email{riikka.korte@aalto.fi}
\author{Nageswari Shanmugalingam}
\address{Department of Mathematical Sciences, P.O.~Box 210025, University of Cincinnati, Cincinnati, OH~45221-0025, U.S.A.}
\email{shanmun@uc.edu}
\thanks{
 L.C.'s work is partially supported by the NSF (U.S.A.) grant DMS~\# 2348806;
 R.G.'s work is partially supported by the NSERC (Canada) grant RGPIN-2025-05905;
 R.K.'s work is partially supported by the Research Council of Finland grant 360184;
 N.S.'s work is partially supported by the NSF (U.S.A.) grant DMS~\#2348748.}
\maketitle


\medskip


\section{Introduction}

In \cite{CGKS, CKKSS}, the authors and their collaborators developed an extension of the approach of Caffarelli and 
Silvestre \cite{CS} for the study of nonlinear
non-local PDE to the setting of doubling metric spaces $(Z, d_Z,\nu_Z)$. 
The Caffarelli-Silvestre approach hinges on the idea that the solutions of certain non-local PDE in $\R^n$ can be realized as critical points of Besov energies, and that such Besov energies are comparable with the Dirichlet energy associated to a Neumann problem for a (local) PDE in $\R^n\times \R^+$. The extension to metric spaces of this idea makes use of hyperbolic fillings to define a metric measure space $(X,d,\mu)$, satisfying both the doubling condition and a Poincar\'e condition, that has $Z$ as its boundary. 
The papers~\cite{CGKS, CKKSS} go further and study nonlocal energies on $Z$ induced as a trace of a uniform domain
when $Z$ arises directly as the boundary of a uniform domain equipped with a doubling measure supporting a 
$p$-Poincar\'e inequality.
In \cite{CGKS,CKKSS} we have proved well-posedness for the Neumann problem in $X$, and inferred properties for the corresponding non-local differential equations on $Z$. 

One aspect of our work was the study of global regularity of 
weak solutions to the Neumann problem for the $p$-Laplacian operator in $X$. 
In the unweighted Euclidean setting the best possible regularity is $C^{1,\alpha}$-smoothness of weak solutions. 
Since,
in our generality the best possible smoothness is H\"older continuity (see \cite{KRS}),
we focused on H\"older regularity of weak solutions up to the boundary. 

Although the hypotheses on the Neumann data $f$ that were needed in \cite{CKKSS} are the same as the ones that arise from the work of Caffarelli and Stinga \cite{CSt} in the Euclidean setting, in the present paper we contribute a different (more general) approach and are able to improve on  the H\"older exponent itself. In particular, we establish H\"older regularity with an exponent that is sharp with respect to the membership of the Neumann data in a Morrey class.

\bigskip

\noindent{\bf Structure  hypotheses:}
Throughout the paper, we fix $1<p<\infty$ and assume that $\Om$ is a bounded domain in
a complete metric measure space $(X,d,\mu)$ such that:
\begin{enumerate}\label{structure hypotheses}
\item[(H0)] $\Omega$ is a John domain.
\item[(H1)] $(\overline \Om, d, \mu\vert_{\overline \Om})$ is doubling and supports a $p$-Poincar\'e inequality.
\item[(H2)]
The boundary $\partial \Omega$ is  complete and uniformly perfect.
Moreover, it
is  equipped with a Radon measure $\nu$ for which there are constants $C\ge 1$
and $0<\Theta<p$ such that for all $x\in \partial\Omega$
and $0<r<2\diam(\partial \Omega)$,
\begin{equation}\label{eq:Co-Dim-intro}
 \frac{1}{C}\, \frac{\mu(B(x,r)\cap\Om)}{r^\Theta}\le \nu(B(x,r)\cap\partial\Om)\le C\, \frac{\mu(B(x,r)\cap\Om)}{r^\Theta};
\end{equation}
that is, $\nu$ is a $\Theta$-codimensional Hausdorff measure with respect to $\mu\lvert_\Om$.
\end{enumerate}

Going forward, the ambient metric measure space $X$ plays no role, and so 
we may take $X=\overline\Om$, in which case every ball $B\subset X$ is automatically a subset of $\overline\Om$. Equivalently, considering $\overline\Om$ to be a subset of $X$, for $x\in \overline\Om$ and $r>0$, we shall interpret the notation $B(x,r)$ to mean $\{y\in\overline\Om:d(x,y)<r\}$. 

Throughout the paper we will assume that hypotheses (H0), (H1) and (H2) above hold.
The constants associated with the conditions (H0), (H1), and (H2), together with the exponent $p$, 
will be referred to as the structural constants. 

\begin{remark}\label{rem:uniform}
Since our main concern is regularity near the boundary, and the proofs are local in nature, 
our results also hold even when $\Om$ is unbounded (though only locally in that case), provided 
that $\Om$ is a uniform domain.
In the situation where $\Om$ is unbounded,
we should also replace the Newton-Sobolev space $N^{1,p}(\Om)$ with the Dirichlet space $D^{1,p}(\Om)$ (see \cite{CGKS} for more details).
\end{remark}

In our previous work \cite{CGKS, CKKSS}, we established the following global regularity result \cite[Theorem 1.6]{CKKSS} for weak solutions $u\in N^{1,p}(\Om)$ of the Neumann problem in $\Omega$, with Neumann boundary data $f\in L^{p'}(\partial \Omega, d\nu)$  (where $p':=p/(p-1)$ denotes the H\"older
conjugate of $p$):
\begin{equation}\label{eq:Neumann}
\int_{\overline{\Om}}|\nabla u|^{p-2}\, \nabla u\, \cdot\, \nabla v\, d\mu=\int_{\overline{\Om}} v\, f\, d\nu
\end{equation} 
for all $v\in N^{1,p}(\Omega)$. 

\begin{thm}[\protect{\cite[Theorem~1.6]{CKKSS}}]
Let $Q^\partial_\mu$ denote the lower mass bound exponent
associated with the doubling measure $\mu$ for balls centered at points in $\partial \Omega$, 
as defined in~\eqref{eq:lower-mass-exp-1}.
 Assume that $1<p\le Q^\partial_\mu$, and let $B_{R_0}$
be a ball of radius $R_0>0$ centered at a point in $\partial \Omega$. If the  boundary data 
satisfies the additional integrability assumption $f\in L^q(B_{2R_0}\cap \partial \Omega, d\nu)$ for some $q$ with
\begin{equation}\label{oldrange-q}
q_0:=\tfrac{Q_\mu^\partial-\Theta}{p-\Theta}<q \le \infty,
\end{equation}
then any 
solution of the Neumann problem $u$ is $\eps$-H\"{o}lder continuous in $B_{R_0}$ with
\begin{equation}\label{old-exponent}
\eps=\min\left\{ \tau, \left(1-\frac{\Theta}{p}\right)\left(1-\frac{q_0}{q}\right)\right\}
 =\min\left\{ \tau, \frac{q(p-\Theta)-Q_\mu^\partial+\Theta}{pq}\right\},
\end{equation}
where ${\tau}>0$ is the H\"older exponent for the interior regularity estimates 
established in~\cite[Theorem~5.2]{KinSh}.  
\end{thm}

For the analogue of the above regularity result in the case where $\Om$ is unbounded, see \cite[Theorem~1.10]{CGKS}. 

In the present paper we improve on this regularity result, providing a better H\"older exponent that is sharp with respect to the hypotheses we require from the Neumann data $f$,
see Remark~\ref{rem:improvement}.
In particular, we consider membership of the Neumann data in an appropriate Morrey space, see Definition~\ref{def:Morrey}. Our main result is Theorem~\ref{new regularity}, given next.
The first part of this theorem follows from Proposition~\ref{prop:Luca-Campanato} and from 
Corollary~\ref{Cor: Morrey type estimate},
while the  second part of the theorem is based on the earlier work 
in~\cite[Lemma~4.8]{BMcSh} (see also \cite{Mak-Holder}). 

We continue to assume the structural hypotheses mentioned above.

\begin{thm}\label{new regularity} 
Let $u\in N^{1,p}(\Omega,\mu)$ be a weak solution 
to~\eqref{eq:Neumann} with $f\in L^{p'}(\partial\Om, \nu)$. 

\begin{enumerate}
\item[1)] If $f\in M^{1,-(\alpha+\Theta)}(\partial\Om,\nu)$ and $-p<\alpha<-p(1-\tau) -\tau$, then $u$ is H\"older continuous in $\overline \Omega$ with H\"older exponent $\frac{p+\alpha}{p-1}\in (0,1)$.
\item[2)] If $f$ does not change sign in a ball $B_{4R}$ centered at a boundary point, and $u$ is H\"older continuous on $\overline{\Omega}\cap B_{4R}$ with H\"older exponent $\lambda$ such that $0<\lambda<\tfrac{p-\Theta}{p-1}$, then 
$f\in M^{1, -(\lambda(p-1) -p+\Theta)}(\partial\Om \cap B_R,\nu)$. 
\end{enumerate}
\end{thm}

To compare the two parts of the above theorem, in the second part of the theorem we set $\alpha:=(p-1)\lambda-p$,
and note then that $\tfrac{p+\alpha}{p-1}=\lambda$ and 
$\lambda(p-1)-p-\Theta=\alpha+\Theta$, and so the two Morrey spaces coincide with that choice of
$\alpha$.

In the second part of the above theorem, in the case where $\lambda>1$, the result becomes trivial as then $u$ is constant and then the Neumann data of a constant function is zero. 

The first part of Theorem~\ref{new regularity}  is a consequence of
Theorem~\ref{Morrey type estimate}
related to weak  solutions $u\in N^{1,p}(\Omega, \mu)$ of the more general equation given by
\begin{equation}\label{eq:Neumann-bar}
\int_\Omega|\nabla u|^{p-2}\, \nabla u\, \cdot\, \nabla v\, d\mu=\int_\Omega v\, d{\overline\nu}
\end{equation}
for all $v\in N^{1,p}(\Omega)$, where $\overline{\nu}$ is a signed Radon measure 
on $\overline{\Om}$ 
with 
$\barnu \in (N^{1,p}(\Om, d\mu))^*$  such that its total variation 
$|\overline{\nu}|=\barnu^+ + \barnu^-$ satisfies $\overline{\nu}(\overline{\Om})=0$ and
\begin{equation} \label{eq:nu-f-growth}
\frac{ |\overline\nu|(B(x,r))}{\mu(B(x,r))}\le M\, r^\alpha
\end{equation}
for some $\alpha<0$ and $M>0$, and for all $x\in \partial\Om$ and $0<r\le R_0$.
Here, $\barnu=\barnu^+-\barnu^-$ is the Hahn decomposition of the signed measure $\barnu$.

{ As before, we denote by $\tau\in (0,1)$ the H\"older exponent for the interior regularity estimates
established in~\cite[Theorem~5.2]{KinSh}.}

\begin{thm}\label{Morrey type estimate}
Let $z_0\in\overline\Om$, $R>0$, 
and let $u\in N^{1,p}(B(z_0,2R))$ satisfy equation~\eqref{eq:Neumann-bar} 
for all
$v\in N^{1,p}(B(z_0,2R))$ with compact support contained in $B(z_0,2R)$,
and suppose that  $|\overline{\nu}|$ satisfies~\eqref{eq:nu-f-growth}. 
If  $-p<\alpha<-p(1-\tau) -\tau$,
then $|\nabla u| \in M^{p, \frac{1+\alpha}{1-p}}(B(z_0,R))$, and consequently,
$u$ is locally $\tfrac{p+\alpha}{p-1}$-H\"older continuous in $B(z_0,R/2)$.
\end{thm}

Theorem~\ref{Morrey type estimate} will be proved in Section \ref{proofs}. 
As an immediate consequence of this theorem, we obtain the following. 

\begin{cor}\label{Cor: Morrey type estimate}
Let $u\in N^{1,p}(\Omega,\mu)$ be a weak solution to \eqref{eq:Neumann} with 
$f\in L^{p'}(\partial\Om, \nu)\cap M^{1,-(\alpha+\Theta)}(\partial\Om,\nu)$.
Then $|\nabla u| \in M^{p, \frac{1+\alpha}{1-p}}(\overline{\Om})$ 
and so $u$ is $\tfrac{p+\alpha}{p-1}$-H\"older continuous in $\overline\Om$
whenever
$-p<\alpha<-p(1-\tau) -\tau$.
\end{cor}

\begin{remark}
If $f\in L^{p'}(\partial\Om, \nu)\cap M^{1,-(\alpha+\Theta)}(\partial\Om,\nu)$ for some $\alpha\geq -p(1-\tau)-\tau$, 
then necessarily $f\in M^{1,-(\beta+\Theta)}(\partial\Om,\nu)$ for all $\beta<\alpha$, and so by choosing 
$\beta<-p(1-\tau)-\tau$ appropriately, we obtain that $u$ is $(\tau-\iota)$-H\"older continuous in 
$\overline{\Om}$ for any $\iota>0$. However, our proof still does not yield that $u$ is 
$\tau$-H\"older continuous up to the boundary.
\end{remark}

The first instance in the literature, to our knowledge, that addresses signed Radon measures as 
non-homogeneous data for the $p$-Laplacian is the work of Ono~\cite{Ono} in the Euclidean setting. 
Our Theorem~\ref{Morrey type estimate} extends to the setting of metric measure spaces the 
work in~\cite{Ono}. The challenge here is that instead of the Ahlfors regularity of Lebesgue measure considered in~\cite{Ono},
we have to contend with knowing only that the measure $\mu$ is doubling.
For more regular measures $\overline\nu$ in the Euclidean setting, corresponding to $\alpha>-1$, 
the work of \cite{DZ} yields gradient estimates for which there are counter-examples in our more general setting, see \cite{KRS}. 
As we are aiming for lower-order regularity, our hypotheses allow for measures that are significantly more singular. 

\begin{remark}\label{rem:improvement}
To better appreciate these results,  we turn our attention briefly to the restriction placed on $f$ 
in~\cite{CGKS, CKKSS} in obtaining H\"older continuity of solutions and show that such hypotheses, 
coupled with the results in the present paper, lead to a sharper H\"older regularity exponent for the weak 
solutions of the Neumann problem.

If $f\in L^q(\partial\Om,\nu)$, then
 $f$ is in the Morrey space $M^{1,-(\alpha+\Theta)}(\partial\Om,\nu)$ for a suitable $\alpha$ (see Definition~\ref{def:Morrey})
 and so Corollary~\ref{Cor: Morrey type estimate} follows from Theorem~\ref{Morrey type estimate} above.
Indeed, by H\"older's inequality, for $x\in\partial\Om$ and $r>0$, we have
\begin{align*}
\int_{B(x,r)}|f|\, d\nu&=
\int_{\partial\Om} \chi_{B(x,r)}(y)\, |f(y)|\, d\nu(y)\\
  &\le 
  \left(\int_{\partial\Om}|f|^q\, d\nu\right)^{1/q}\, \nu(B(x,r))^{1/q'}\\
  &\approx \|f\|_{L^q(\partial\Om,\nu)}\, r^{-\Theta/q'}\, \mu(B(x,r))^{1/q'}.
\end{align*} 
In the last step above, we used the $\Theta$-codimentionality property~\eqref{eq:Co-Dim-intro} from~(H2). Now using
the lower mass bound property of $\mu$ from~\eqref{eq:lower-mass-exp-1}, we see that when $\xi\in\partial\Om$, $r<R_0$,
and $x\in\partial\Om\cap B(\xi,R_0)$, 
\begin{align*}
\frac{1}{\mu(B(x,r))}\int_{B(x,r)}|f|\, d\nu
&\le C\, \frac{\|f\|_{L^q(\partial\Om,\nu)}}{\mu(B(x,r))^{1/q}}\, r^{-\Theta/q'}\\
&\le C\, \frac{\|f\|_{L^q(\partial\Om,\nu)}}{\mu(B(\xi, R_0))^{1/q}}\, r^{-\Theta/q'}\, \left(\frac{R_0}{r}\right)^{Q^\partial_\mu/q}\\
   &=C\, \frac{R_0^{Q^\partial_\mu/q}\, \|f\|_{L^q(\partial\Om,\nu)}}{\mu(B(\xi, R_0))^{1/q}}\, r^{-(Q^\partial_\mu/q+\Theta/q')}.
\end{align*}
Thus, $f$ is in the Morrey class $M^{1,-(\alpha+\Theta)}(\partial\Om,\nu)$
with the choice of 
\[
\alpha= - \frac{Q^\partial_\mu+(q-1)\Theta}{q}.
\]
In view of Theorem \ref{new regularity}, we know that 
if we have
$-p<\alpha<-p(1-\tau) -\tau$, then
we obtain $\tfrac{p+\alpha}{p-1}$-H\"older continuity of weak solutions to the Neumann problem.  In our context, this implies the bound 
\[
q>\frac{Q^\partial_\mu-\Theta}{p-\Theta}.
\]
In our setting, this bound with the Theorem \ref{new regularity} 
implies that the range \eqref{oldrange-q} guarantees that the solutions are H\"older continuous up to the boundary with H\"older exponent
\begin{equation}\label{sharp exponent}
\min\left\{\tau,\frac{p+\alpha}{p-1}\right\} = \min\left\{\tau,\frac{q(p-\Theta)-Q^\partial_\mu+\Theta}{q(p-1)}\right\}>\eps,
\end{equation}
where $\eps$ is the H\"older exponent derived in \cite{CKKSS}, as in \eqref{old-exponent}.
Thus, the results of the present note improves the H\"older regularity obtained in~\cite{CGKS,CKKSS} when 
$q$ satisfies the above conditions.
\end{remark}

\bigskip

Finally,
we apply the 
results on the Neumann problem to infer this sharper H\"older regularity of solutions to fractional powers of the 
$p$-Laplacian in doubling metric measure spaces, see Theorem~\ref{nonlocal}.

\begin{remark}
In the Euclidean literature on regularity of weak solutions of elliptic equations of the form $Lu=f$, the first time 
the emphasis switched from $L^p$ conditions on $f$ to integral growth conditions on $f$ was in 
Morrey's papers~\cite{Mor2, Mor}. 
\end{remark}

\section{Preliminary results and definitions}

In this section, we provide some of the basic definitions and results from the literature that will play a role in our proofs.

\subsection{Doubling property and codimensionality}\label{subsect:doubling}
A measure $\mu$ is said to be \emph{doubling} if it is a Radon measure and there is a constant $C_d\ge 1$ such that
\[
0<\mu(B(x,2r))\le C_d\, \mu(B(x,r))<\infty
\]
for each $x\in X$ and $r>0$.
Doubling measures satisfy a lower mass bound property: {\it there are constants $c>0$
and $Q_\mu>0$, depending only on $C_d$,  such that
for each $x\in X$, $0<r<R<2\, \diam(\Om)$, and for each $y\in B(x,R)$,
\begin{equation}\label{eq:lower-mass-exp}
c\left(\frac{r}{R}\right)^{Q_\mu}\le \frac{\mu(B(y,r))}{\mu(B(x,R))},
\end{equation}
see~\cite[page~76]{HKST}. } 
In the special case when $\Omega\subset X$ is a John domain and $\overline\Omega=X$, we use 
the symbol $Q_\mu^\partial$ to denote the lower mass bound exponent obtained by considering only balls 
$B(x,r), B(x,R)$ with $x\in \partial \Omega$.  
That is, when $x\in\partial\Om$ and $0<r<R<2\, \diam(\Om)$, 
\begin{equation}\label{eq:lower-mass-exp-1}
c\left(\frac{r}{R}\right)^{Q_\mu^\partial}\le \frac{\mu(B(x,r))}{\mu(B(x,R))}.
\end{equation}
In general, it is possible to have $Q_\mu^\partial < Q_\mu$, see Remark~\ref{remark:CSexample}.

 Given a Radon measure $\mu$ on a domain $\Om$ and a Radon measure $\nu$ on $\partial\Om$,
we say
that $\nu$ is \emph{$\Theta$-codimensional with respect to $\mu$} if for some constant $C\ge 1$
we have 
\begin{equation}\label{eq:codim}
\frac{1}{C}\, \frac{\mu(B(\xi,r)\cap\Om)}{r^\Theta}\le \nu(B(\xi,r)\cap\partial\Om)\le C\, \frac{\mu(B(\xi,r)\cap\Om)}{r^\Theta}
\end{equation}
whenever $0<r<2\diam(\partial\Om)$ and $\xi\in\partial\Om$.
We extend $\mu$ to $\overline{\Om}$ by setting $\mu(\partial\Om)=0$, whence it follows that $\mu$ is 
doubling on $\overline{\Om}$. In the above
codimensionality condition we will, henceforth, dispense with the intersection of $B(\xi,r)$ with $\Om$.
Similarly, we extend $\nu$ to $\overline{\Om}$ by setting $\nu(\Om)=0$. 

\subsection{Sobolev-type spaces and Poincar\'e inequalities}\label{subsect:sob-PI}
One of the main features of a first-order calculus in metric measure spaces was first developed by
Heinonen and Koskela~\cite{HK} by using a notion of upper gradients as a substitute for weak derivatives.
Given a measurable function $u:X\to\R$, we say that a non-negative Borel function $g$ on
$X$ is an \emph{upper gradient} of $u$ if
\[
 |u(x)-u(y)|\le \int_\gamma g\, ds
\]
for every non-constant compact rectifiable curve $\gamma$ in $X$. Here, $x$ and $y$ denote the terminal points of $\gamma$.

The function $u$ is said to be in the \emph{homogeneous Sobolev space} $D^{1,p}(X)$ 
if $u$ has an upper gradient that belongs to $L^p(X)$;
and, $u$ is said to be in the \emph{Newton-Sobolev class} $N^{1,p}(X)$ if it is in $D^{1,p}(X)$ and, in addition, 
$\int_X|u|^p\, d\mu$ is finite.
Given that upper gradients are not unique, we define the energy seminorm on $D^{1,p}(X)$ by
\begin{equation}\label{penergy}
\mathcal{E}_p(u)^p:=\inf_g\int_Xg^p\, d\mu,
\end{equation}
where the infimum is over all upper gradients $g$ of $u$. The norm on $N^{1,p}(X)$ is given by
\[
\Vert u\Vert_{N^{1,p}(X)}:=\Vert u\Vert_{L^p(X)}+\mathcal{E}_p(u).
\]

If $1\le p<\infty$, then for each $u\in D^{1,p}(X)$ there is a unique (up to sets of $\mu$-measure zero) non-negative 
function $g_u$
that is the $L^p$-limit of a sequence of upper gradients of $u$ from $L^p(X)$ and so that for each upper 
gradient $g$ of $u$ we
have that $\Vert g_u\Vert_{L^p(X)}\le \Vert g\Vert_{L^p(X)}$. The functions $g_u$ belong to a larger class 
of ``gradients lengths'' of $u$,
called $p$-weak upper gradients, see for example~\cite{BBbook, HajK, HKST, Sh}. This 
function $g_u$ is said to be the \emph{minimal $p$-weak upper gradient} of $u$.

For $1\le p<\infty$, the metric measure space $(X,d,\mu)$ is said to support a $p$-Poincar\'e inequality if there are constants
$C_P>0$ and $\lambda\ge 1$ such that for all $u\in N^{1,p}(X)$ and balls $B=B(x,r)$ in $X$, we have,
with $u_B:=\mu(B(x,r))^{-1}\, \int_{B(x,r)} u\, d\mu$,
\[
\jint_{B(x,r)}|u-u_B|\, d\mu\le C_P\, r\, \left(\jint_{B(x,\lambda r)}g_u^p\, d\mu\right)^{1/p}.
\]
As shown in~\cite{HajK}, if
$X$ is a length space, then we can take $\lambda=1$ at the expense of increasing the constant
$C_P$. It is also well known that the $p$-Poincar\'e inequality and the doubling property of the measure
together imply the $(p,p)$-Poincar\'e inequality
 as follows:
\[
\jint_{B(x,r)}|u-u_B|^p\, d\mu\le C_{SP}\, r^p\, \jint_{B(x,r)}g_u^p\, d\mu,
\]
see for instance~\cite{HajK, HKST}.

Related to the class $N^{1,p}(U)$ for a given domain $U\subset X$, is the notion of 
\emph{Newton-Sobolev
spaces with zero boundary values}; these are crucial in posing Dirichlet boundary value problems. Given a domain
$U\subset X$ with $X\setminus U$ non-empty, we say that a function $f\in N^{1,p}(X)$ is in 
$N^{1,p}_0(U)$
if there is a representative $\widehat{f}$ of $f$ in $N^{1,p}(X)$ such that $\widehat{f}=0$ pointwise everywhere in
$X\setminus U$. We have that $N^{1,p}_0(U)$ consist of the $N^{1,p}(X)$-norm closure of the collection of all
functions in $N^{1,p}(X)$ with compact support contained in $U$; see Theorem 4.8 in \cite{Sh2}. 

\subsection{ John and uniform domains} \label{subsect:John}
A domain $\Om$ in a complete metric space $X$ is said to be a {\it John domain} if there is a point $x_0\in\Om$, called a  {\it John center},
and a  constant $C_J\ge 1$ such that whenever $x\in\Om$, there exists a rectifiable curve $\gamma_x$
in $\Om$ with end points $x_0$ and $x$ such that for each point $z$ in the image of $\gamma_x$, we have
that
\[
 \dist_{X\setminus\Om}(z)\ge C_J^{-1} \ell(\gamma_x[x,z]),
\]
where $\gamma_x[x,z]$ denotes the segments of $\gamma_x$ with end points $z$ and $x$. As a consequence,
a John domain is a connected open set and, moreover, if $\Om\ne X$ then $\Om$ is bounded.  

As mentioned in Remark~\ref{rem:uniform}, when $\Om$ is unbounded
it cannot be a John domain. A sufficient replacement for a localized version of our result is to assume $\Om$ to be a {\it uniform domain}.
Uniform domains are
characterized by the existence of a constant
$C_U\ge 1$ such that for every pair $x,y\in \Om$ there exists a rectifiable curve $\gamma_{xy}$ joining them with the property
\[
 \dist_{X\setminus\Om}(z)\ge C_U^{-1}  \min\bigg\lbrace\ell(\gamma_{xy} [x,z]),
 \ell(\gamma_{xy}[z,y])\bigg\rbrace \ \text{ and } \ \ell(\gamma_{xy}) \le C_U d(x,y),
\]
for all $z\in \gamma_{xy}$.  It is immediate to see that a bounded uniform domain is also John, but the converse is false
as demonstrated by the example of a planar slit disk, which is a John domain but is not a uniform domain.

\subsection{Differentiable structures}\label{Sec:Cheegerishness}
Some of the properties we are interested in depend on the existence of an Euler-Lagrange equation satisfied 
by energy minimizers. 
To achieve that we use a
Cheeger differentiable structure (see~\cite{Che}).
A metric measure space $(\Om,d,\mu)$ is said to support a Cheeger differential structure of dimension $N\in\N$ if there exists a 
collection of coordinate patches $\{(\Om_\alpha,\psi_\alpha)\}$ and a $\mu$-measurable inner product 
$\langle \cdot,\cdot\rangle_{x}$, $x\in \Om_\alpha$, on $\R^N$ such that
\begin{enumerate}
\item each $\Om_\alpha$ is a measurable subset of $\Om$ with positive measure and 
$\bigcup_\alpha \Om_\alpha$ has full measure;
\item each $\psi_\alpha:\Om_\alpha\rightarrow\R^N$ is Lipschitz;
\item for every function $u\in D^{1,p}(\Om)$, for $\mu$-a.e. $x\in \Om_\alpha$ there is a vector $\nabla u(x)\in\R^N$ such that
\[
	\elimsup\limits_{\Om_\alpha\ni y\rightarrow x}\frac{|u(y)-u(x)-\langle\nabla u(x),\psi_\alpha(y)-\psi_\alpha(x)\rangle_{x}|}{d(y,x)}=0.
\]
\end{enumerate}

When the metric $d$ is doubling, we may assume that the collection of coordinate patches is countable and 
that the coordinate 
neighborhoods $\{\Om_\alpha\}$ are pairwise disjoint. Note that there may be more than one possible 
Cheeger differential structure on a given space. From~\cite{Che} we also know that we can
choose the inner product structure $\langle\cdot,\cdot\rangle_x$ so that there is a constant $c>0$ 
with the property that when $u\in D^{1,p}(\Om)$,
\[
\frac{g_u(x)^2}{c}\le |\nabla u(x)|^2_x=\langle\nabla u(x),\nabla u(x)\rangle_x\le c\, g_u(x)^2
\]
for $\mu$-a.e.~$x\in X$. Thus, in the Poincar\'e inequalities, we can replace the quantity 
$\int_B g_u^p\, d\mu$ with the quantity
$\int_B|\nabla u(x)|_x^p\, d\mu(x)$.

A function $u\in N^{1,p}(U)$, where $U$ is relatively open in $\overline{\Om}$, is a \emph{(Cheeger) $p$-harmonic
function} in $U$ if, whenever $v\in D^{1,p}(U)$ has compact support in $U$, we have
\[
\int_{\text{supt}(v)}|\nabla u|^p\, d\mu\le \int_{\text{supt}(v)}|\nabla (u+v)|^p\, d\mu.
\]
Equivalently, we have the following corresponding Euler-Lagrange equation:
\[
\int_U|\nabla u(x)|^{p-2}\langle \nabla u(x),\nabla v(x)\rangle_x\, d\mu(x)=0.
\]
We say that $u$ is a solution to the Dirichlet problem on a ball $B\subset\overline\Om$ with the same boundary values as $w\in N^{1,p}(B)$ if $u$ is $p$-harmonic in $B$ and $u-w\in N^{1,p}_0(B)$. 

For brevity, in our exposition we will suppress the dependence of $x$ on the inner product structure, and denote
\[
\langle \nabla u(x),\nabla v(x)\rangle_x=:\nabla u(x)\cdot\nabla v(x),
\]
with $\nabla u(x)\cdot\nabla u(x)$ also denoted by $|\nabla u(x)|^2$,
when this will not lead to confusion.
Cheeger $p$-harmonic functions are quasiminimizers of the 
$p$-energy~\eqref{penergy} in the sense of Giaquinta~\cite{Gia}, 
and hence
we can avail ourselves of the properties derived in~\cite{KinSh}.

\subsection{Trace and extension theorems}\label{s:trace}
For a uniform domain $\Om$ with bounded boundary $\partial \Omega$, 
the existence of bounded linear trace operators 
$T:D^{1,p}(\Omega,\mu)\rightarrow B^{1-\Theta/p}_{p,p}(\partial\Omega,\nu)$ was established 
in \cite[Proposition 8.3]{GKS} and follows from the earlier work of Mal\'y \cite{Mal} for John 
domains.
Here, $B^{1-\Theta/p}_{p,p}(\partial\Omega,\nu)$ is a Besov space, see Section~\ref{sec:4} for the definition.
We recall that the trace operator $Tu:\partial\Om\to\R$ is given by
\[
\lim_{r\to 0^+}\frac{1}{\mu(B(\xi,r)\cap\Om)}\, \int_{\mu(B(\xi,r)\cap\Om)}|u-Tu(\xi)|\, d\mu=0
\]
for $\nu$-almost every $\xi\in\partial\Om$.

\begin{remark}
In this paper, we assume that $\mu$ is doubling but not necessarily Ahlfors regular; it follows then from the 
codimensionality condition~\eqref{eq:Co-Dim-intro} above that the measure $\nu$ on $\partial\Om$ is also 
doubling on $\partial\Om$,
even though it is not doubling on $\overline{\Om}$. We also know from the $\Theta$-codimensionality of
$\nu$ with respect to $\mu$ that a set of $p$-capacity zero is necessarily of $\nu$-measure zero, see~\cite{GKS} for
instance. From~\cite{KinLat} it follows that $p$-capacity almost every point is a Lebesgue point of
a Newton-Sobolev function, and so the above definition of trace also sets the value of the trace at $\nu$-almost
every point in $\partial\Om$.
\end{remark}

\subsection{Potential-theoretic preliminaries}

In this subsection, we gather together some theorems 
 that we use in proving the results of
the present paper.
We start with the following weak version of the Maz'ya capacitary 
inequality. The result, in the metric setting, can be found in~\cite[Lemma~2.1]{KinSh}. 

\begin{lem}\label{lem:Mazya}
	Let $(Y,d_Y,\mu_Y)$ be a compact doubling metric measure space supporting a $p$-Poincar\'e inequality.
For each $x\in Y$ and $0<r<\tfrac{1}{4}\diam(Y)$,  there is a constant $C\ge 1$ 
that depends on the constants associated with the doubling property of $\mu_Y$ and the $p$-Poincar\'e inequality
such that 
for each $u\in N^{1,p}_0(B(x,r))$ we have that
\[
\int_{B(x,r)}|u|^p\, d\mu_Y\le C\, r^p\, \int_{B(x,r)}|\nabla u|^p\, d\mu_Y.
\]
\end{lem}

This inequality has the following generalization, called an Adams-type inequality.
Notice that by \cite{KeithZhong}, the $p$-Poincar\'e inequality, together with the fact that $1<p<\infty$,
implies that the space $X$
supports also a $t$-Poincar\'e inequality with some $1\leq t<p$. 

\begin{lem}\label{lem:Adams t}
Suppose that $(X,d,\mu)$ satisfies a $t$-Poincar\'e inequality 
for some $1\le t<p$
and that $|\overline{\nu}|$ 
satisfies~\eqref{eq:nu-f-growth}. Then, for $p^{*}:=tp(Q_{\mu}+\alpha)/(t Q_{\mu} -p)$ with
$Q_\mu$ from~\eqref{eq:lower-mass-exp}, there 
exists a constant $C>0$ such that for all $x_0\in \overline{\Omega}$ and $R>0$,
\[
\left( \frac{1}{ \mu(B(x_{0},R))\mathcal{M}R^{\alpha} } \int_{B(x_0,R)} |w|^{p^{*}} d |\overline{\nu}|
\right)^{1/p^{*}}
\le C R \left(\jint_{B(x_{0},R)}|\nabla w|^{p} d\mu \right)^{1/p}
\]
whenever $w\in N^{1,p}_0(B(x_0,R))$.
\end{lem}

Our formulation of Lemma~\ref{lem:Adams t} follows from ~\cite[Theorem 1.4]{Mak-Adams} by noticing that 
$\tfrac{t-1}{t}+\tfrac{Q_{\mu}}{p}-\tfrac{Q_{\mu}}{p^{*}}=1+\frac{\alpha}{p^{*}}$ and re-organizing the terms. 

\bigskip

Next we turn to properties of (Cheeger) $p$-harmonic functions. We start by recalling  the following 
result from~\cite[Proposition~3.3, Proposition~4.3, Theorem~5.2]{KinSh}. 

\begin{lem}\label{lem:harm-Holder}
Suppose that $v\in N^{1,p}(B(x_0,2R))$ is $p$-harmonic in $B(x_0,2R)$. Then, for $0<r<R$ and $k\in\R$, we have that
\begin{equation}\label{eq:DeGiorgi}
\int_{B(x,r)}|\nabla v|^p\, d\mu\le \frac{C}{(R-r)^p}\, \int_{B(x,R)}|v-k|^p\, d\mu,
\end{equation}
\begin{equation}\label{eq:weak-Harnack}
\osc_{B(x,r)}v\, \le C\, \left(\jint_{B(x,2r)}|v-v_{B(x,2r)}|^p\, d\mu\right)^{1/p},
\end{equation}
and
\begin{equation}\label{eq:Holder}
\osc_{B(x,r)} \, v\le C\, \left(\frac{r}{R}\right)^\tau\, \osc_{B(x,R)}\, v,
\end{equation}
with $0<\tau\le 1$ and $C>1$ depending solely on the constants associated with the doubling property of $\mu$ and the
constants associated with the $p$-Poincar\'e inequality.
\end{lem}

In the above lemma, for $A\subset X$ we set 
\[
\osc_A\, v:=\sup\{|v(y)-v(x)|\, :\, x,y\in A\}.
\]
Thanks to the above lemma, we have the following decay estimates for gradients of $p$-harmonic functions on balls,
see also~\cite[Lemma~3.10]{Mak-Holder}. 

\begin{lem}\label{lem:harm-energy-decay}
Suppose that $v\in N^{1,p}(B(x_0,2R))$ is $p$-harmonic in $B(x_0,2R)$. Then for $0<r\le R/4$ we have that
\[
\jint_{B(x,r)}|\nabla v|^p\, d\mu\le C\, \left(\frac{r}{R}\right)^{\tau p-p}\, \jint_{B(x,R)}|\nabla v|^p\, d\mu.
\]
\end{lem}

\begin{proof}
From~\eqref{eq:DeGiorgi} with $2r$ playing the role of $R$ there, we have that
\[
\int_{B(x,r)}|\nabla v|^p\, d\mu\le \frac{C}{r^p}\, \int_{B(x,2r)}|v-v(x)|^p\, d\mu.
\]
An application of~\eqref{eq:Holder} and the doubling property of $\mu$ now gives
\[
\int_{B(x,r)}|\nabla v|^p\, d\mu\le \frac{C}{r^p}\, \left(\osc_{B(x,2r)} v\right)^p\, \mu(B(x,r))
 \le \mu(B(x,r))\, \frac{C}{r^p}\, \left(\frac{r}{R}\right)^{\tau p}\, \left(\osc_{B(x,R/2)}\, v\right)^p.
\]
Now an application of~\eqref{eq:weak-Harnack} gives
\[
\jint_{B(x,r)}|\nabla v|^p\, d\mu\le \frac{C}{r^p}\, \left(\frac{r}{R}\right)^{\tau p}\, \jint_{B(x,R)}|v-v_{B(x,R)}|^p\, d\mu.
\]
Finally, an application of the $(p,p)$-Poincar\'e inequality yields  \[
\jint_{B(x,r)}|\nabla v|^p\, d\mu\le \frac{C}{r^p}\, \left(\frac{r}{R}\right)^{\tau p}\, R^p\, \jint_{B(x,R)} |\nabla v|^p\, d\mu.
\]
A rearrangement of the terms on the right-hand side gives the desired conclusion.
\end{proof}

Let $(Y,d_Y,\mu_Y)$ be a metric measure space. For a compact set $K\subset B(x,r)$, where $x\in Y$ and $r>0$, the relative $p$-capacity $\rcapa_p(K, Y\setminus B(x,2r))$ is the number
\[
\rcapa_p(K, Y\setminus B(x,2r))=\inf_u\int_Y\, g_u^p\, d\mu_Y,
\]
where the infimum is over all $u\in N^{1,p}(Y)$ that satisfy $u\ge 1$ on $K$ and $u=0$ on $Y\setminus B(x,2r)$. The following lemma establishes a uniform $p$-fatness condition (in the sense of Lewis~\cite{JohnLewis}, see~\cite[Definition~1.1]{BMcSh}
for the metric setting) for subsets of $Y$ that have positive codimension. 

\begin{lem}\label{lem:Loewner}
	Let $(Y,d_Y,\mu_Y)$ be a compact doubling metric measure space supporting a $p$-Poincar\'e inequality, and 
	$E\subset Y$ be a closed set supporting a Borel measure $\nu$. If there are constants $C\ge 1$ and $0<\Theta<p$
	such that for each $x\in E$ and $0<r\le \diam(E)$ we have
	\[
	\frac{1}{C}\, \frac{\mu_Y(B(x,r))}{r^\Theta}\le \nu(B(x,r))\le C\, \frac{\mu_Y(B(x,r))}{r^\Theta},
	\]
	then there is a constant $\Lambda>0$ such that for each $x\in E$ and $0<r<\tfrac{1}{4}\diam(E)$, we have
	\[
	\frac{\rcapa_p(\overline{B}(x,r)\cap E, Y\setminus B(x,2r))}{\rcapa_p(\overline{B}(x,r), Y\setminus B(x,2r))}
	\ge \Lambda.
	\]
\end{lem}

\begin{proof}
	From~\cite[Lemma~2.6]{BMcSh} (or~\cite[Lemma~3.3]{Bj}),
	we know that 
	\[
	\rcapa_p(\overline{B}(x,r), Y\setminus B(x,2r))\approx \frac{\mu(B(x,r))}{r^p}.
	\]
	Thus it suffices to show that 
	\[
	\rcapa_p(\overline{B}(x,r)\cap E, Y\setminus B(x,2r))\gtrsim\frac{\mu(B(x,r))}{r^p}.
	\]
	The proof, given here, uses a technique developed in~\cite{HK}. 
	
	Fix $x_0\in E$ and $0<r< \tfrac{1}{4} \diam(E)$.
	Let $u\in N^{1,p}(Y)$ such that $u\ge 1$ on $\overline{B}(x_0,r)\cap E$ and $u=0$ on $Y\setminus B(x_0,2r)$. Then,
	for each $x\in \overline{B}(x_0,r)\cap E$ and $y\in Y\setminus B(x_0,2r)$ we have that $|u(x)-u(y)|\ge 1$. 
	Therefore by the triangle inequality one of two cases
	must occur: either, for every $x\in \overline{B}(x_0,r)\cap E$ we have that
	$|u(x)-u_{B(x_0,4r)}|>\tfrac13$, or else, for every $y\in Y\setminus B(x_0,2r)$ we must have that
	$|u(y)-u_{B(x_0,4r)}|>\tfrac13$. 
	
	Let us consider first the first case that  
	$|u(x)-u_{B(x_0,4r)}|>\tfrac13$ for each $x\in \overline{B}(x_0,r)\cap E$. We know that $\nu$-a.e. point in $E$ is a Lebesgue point, see for instance~\cite[Propositions~3.11 and~8.3]{GKS}. See also \cite{KinLat, HKST} for more on Lebesgue-point properties of Sobolev functions. Thus for $\nu$-a.e.  $x\in \overline{B}(x_0,r)\cap E$ 
	we have that 
	
	\begin{align*}
		\frac13<|u(x)-u_{B(x_0,4r)}|\le \sum_{j\in\N}|u_{B_j(x)}-u_{B_{j+1}(x)}|,
	\end{align*}
	where $B_1(x):=B(x_0,4r)$ and for positive integers $j\ge 2$, $B_j(x):=B(x,2^{2-j}r)$. Noting that 
	$B_{j+1}(x)\subset B_j(x)$ for 
	each positive integer $j$, from the doubling property of $\mu_Y$ followed by the Poincar\'e inequality, it follows that
	\begin{align*}
		\frac13\le \sum_{j\in\N}|u_{B_j(x)}-u_{B_{j+1}(x)}|
		\lesssim & \sum_{j\in\N}\jint_{B_j(x)}|u-u_{B_j(x)}|\, d\mu_Y\\
		\lesssim & \sum_{j\in\N} 2^{-j}r\, \left(\jint_{\lambda B_j(x)}g_u^p\, d\mu_Y\right)^{1/p}\\
		\lesssim & \sum_{j\in\N}\frac{2^{-j}r}{\mu_Y(B_j(x))^{1/p}}\, \left(\int_{\lambda B_j(x)}g_u^p\, d\mu_Y\right)^{1/p}.
	\end{align*}
	Since such $x$ are in the set $E$, by the assumption on the measure $\nu$ we have for $\eta>0$,
	\begin{align*}
		\frac13 c(\eta)\, \sum_{j\in\N}2^{-j\eta}=
		\frac13\lesssim & \sum_{j\in\N}\frac{(2^{-j}r)^{1-\Theta/p}}{\nu(B_j(x))^{1/p}}\, \left(\int_{\lambda B_j(x)}g_u^p\, d\mu_Y\right)^{1/p}.
	\end{align*}
	All the comparison constants implicitly referred to above depend solely on the doubling constant and the
	constant associated with the Poincar\'e inequality.
	It follows that there is a positive integer $j_x$ such that
	\[
	\frac{c(\eta)}{3}\, 2^{-j_x\eta}
	\lesssim \frac{(2^{-j_x}r)^{1-\Theta/p}}{\nu(B_{j_x}(x))^{1/p}}\, \left(\int_{\lambda B_{j_x}(x)}g_u^p\, d\mu_Y\right)^{1/p},
	\]
	that is,
	\[
	2^{-j_x(\eta p-p+\Theta)}\nu(B_{j_x}(x))\lesssim c(\eta)^{-p}\, r^{p-\Theta}\, \int_{\lambda B_{j_x}(x)}g_u^p\, d\mu_Y.
	\]
	Choosing $\eta=1-\tfrac{\Theta}{p}>0$ in the above analysis, we get
	\[
	\nu(B_{j_x}(x))\lesssim r^{p-\Theta}\, \int_{\lambda B_{j_x}(x)}g_u^p\, d\mu_Y,
	\]
	where the comparison constant now also depends on $c(\eta)$ corresponding to the choice of $\eta$ made above,
	and so on $p$ and $\Theta$.
	
	The collection $\lambda B_{j_x}(x)$, $x\in E\cap\overline{B}(x_0,r)$ with $x$ a Lebesgue point of $u$, is a cover of
	this set. Thanks to the $5$-covering theorem~\cite{Hei}, we obtain a countable pairwise disjoint subcollection
	$\{B_k\}_{k\in I\subset \N}$ such that $\{5B_k\}_{k\in I}$ is a cover of that set. 
	
	Recall that $\mu_Y$ is doubling. It follows that $\nu$ is also a doubling measure. As the set of points $x\in E$ that
	are not Lebesgue points of $u$ forms a $\nu$-measure zero set, it follows that
	\begin{align*}
		\frac{\mu_Y(B(x_0,r))}{r^\Theta}\lesssim
		\nu(E\cap\overline{B}(x_0,r))\le \sum_{k\in I}\nu(5B_k)\lesssim \sum_{k\in I}\nu(\tfrac{1}{\lambda}B_k)
		&\lesssim r^{p-\Theta}\, \sum_{k\in I}\int_{B_k}g_u^p\, d\mu_Y\\
		&\le r^{p-\Theta}\, \int_Yg_u^p\, d\mu_Y,
	\end{align*}
	that is,
	\begin{equation}\label{eq:ABC}
		\frac{\mu_Y(B(x_0,r))}{r^p}\lesssim \int_Yg_u^p\, d\mu_Y.
	\end{equation}
	
	On the other hand, if for every $y\in Y\setminus B(x_0,2r)$ we have that
	$|u_{B(x_0,4r)}|=|u(y)-u_{B(x_0,4r)}|>\tfrac13$, then as $u(y)=0$ for each $y\in B(x_0,4r)\setminus B(x_0,2r)$,
	by Lemma~\ref{lem:Mazya}, that is the Maz'ya inequality, together with the fact that $0<r<\tfrac{1}{4}\diam(E)\leq\tfrac{1}{4}\diam(Y)$, 
	we have 
	\[
	\tfrac{1}{3}\lesssim C\, r\, \left(\jint_{B(x_0,4r)}\, g_u^p\, d\mu_Y\right)^{1/p}.
	\]
	This again leads to~\eqref{eq:ABC}.
	
	Now, taking the infimum over all such $u$, from~\eqref{eq:ABC} we have
	\[
	\frac{\mu_Y(B(x_0,r))}{r^p}\lesssim \rcapa_p(\overline{B}(x_0,r)\cap E, Y\setminus B(x_0,2r)).
	\]
	As the above holds for all $x_0\in E$ and $0<r<\tfrac{1}{4}\diam(E)$, the claim of the lemma follows.
\end{proof}

\begin{remark}\label{rem:fat}
	Such uniform fatness estimates are useful in establishing boundary regularity of solutions to 
	Dirichlet problems with H\"older continuous boundary data. 
	From~\cite[Theorem~5.1]{BMcSh} (see~\cite[Theorem~3.1]{Danielli} for the setting of H\"ormander
	vector fields in $\R^n$), under the structure hypotheses of the present paper,
	we know that there is some $\delta_F\in(0,1)$, depending on the structural constants and $\Lambda$ 
	from Lemma~\ref{lem:Loewner}, such that
	every $p$-harmonic function in $\Om$ with a $\beta$-H\"older continuous trace on $Z\cap B(\xi, r)$ for some $\beta>0$,
	$r>0$, and $\xi\in Z=\partial\Om$,
	is necessarily $\min\{\beta,\tau,\delta_F\}$-H\"older continuous on $X\cap B(\xi, r/2)$. Here, we recall that ${\tau}>0$ is the H\"older exponent for the interior regularity estimates 
	established in~\cite[Theorem~5.2]{KinSh}.  
\end{remark}

\subsection{Morrey  spaces}\label{subsec-Morrey}
Next we recall the definition of the Morrey  spaces. 
\begin{defn}\label{def:Morrey} 
Let $(Y,d_Y,\mu_Y)$ be a metric measure space with $\mu_Y$ a Borel regular measure on $Y$, 
$\lambda\in \R$, $1\le s <\infty$ and $R_0>0$. 
The Morrey space $M^{s,\lambda} (Y,\mu_Y)$ is defined by
\[
M^{s,\lambda}(Y,\mu_Y)=\Bigg\{ g\in L^s_{loc}(Y,\mu_Y)\, :\,  [g]_{M^{s,\lambda}}^s
:= \sup_{x\in Y, 0<r\le R_0} r^{s \lambda }  \jint_{B(x,r)} |g|^s d\mu_Y <\infty \Bigg\}.
\]
\end{defn}

In the following lemma we show that if, in the definition above, we replace $Y$ with $\partial\Om$ and $\mu_Y$ with the measure $\nu$, then 
the choice of $\nu_f$ given by $d\nu_f=f\, d\nu$, the decay condition~\eqref{eq:nu-f-growth},
and the codimensionality condition~\eqref{eq:Co-Dim-intro}, yield that $f\in M^{1,-(\alpha+\Theta)}(\partial\Om,\nu)$.

\begin{lem}\label{membership-Morrey}
For $f\in L^{p'}(\overline{\Om}, \nu)$, setting  $d|\nu_f|=|f|\, d\nu$, 
the following are equivalent:
\begin{enumerate}
\item the function $f$ is in the Morrey space $M^{1,-(\alpha+\Theta)}(\partial\Om,\nu)$,
\item there is some $\M>0$ such that for each $x\in\overline{\Om}$ 
and $0<r\le R_0$ we have that
\begin{equation}\label{eq:Mor-f}
\frac{1}{\mu(B(x,r))}\int_{B(x,r)}|f|\, d\nu\le \max\{1,2^{\alpha}\}\, C_D^2\, \M\, r^{\alpha},
\end{equation}
\item the measure $|\nu_f|$
satisfies \eqref{eq:nu-f-growth}, i.e. there is some constant $\M>0$ such that
\[
\frac{|\nu_f|(B(x,r))}{\mu(B(x,r))}\le \M\, r^\alpha
\]
for all $x\in \partial\Om$ and $0<r\le R_0$.
\end{enumerate}
\end{lem} 

\begin{proof}
The equivalence between~(3) and~(1) follows from the $\Theta$-codimensionality of $\nu$ as in~(H2). 
Observe also that~(2) implies~(3) by choosing $x\in\partial\Om$ in~\eqref{eq:Mor-f} and $\M$ in~(3) replaced by 
$\max\{1,2^\alpha\}\, C_D^2\, \M$.
Thus,
we devote the remainder of the proof to proving that~(3) implies~(2).

Suppose that~(3) holds.
If $x\in\partial\Om$, then~\eqref{eq:Mor-f} follows from the assumption~(3) and 
$C_D\ge 1$.
Thus it suffices to consider the case $x\in\Om$.
If $B(x,r)$ does not intersect $\partial\Om$, then $\int_{B(x,r)}|f|\, d\nu=0$, and \eqref{eq:Mor-f} follows trivially. Hence,  without loss of generality, we can 
 assume that $x\in \Omega$ and $B(x,r)\cap\partial\Om$ is nonempty. In this case, we can choose $\xi\in B(x,r)\cap\partial\Om$,
and note that then $B(x,r)\subset B(\xi,2r)$. It follows from the assumption~(3)
and the fact that $0<2r\le R_0$ that
\[
\frac{1}{\mu(B(x,r))}\int_{B(x,r)}|f|\, d\nu\le C_D^2\, \frac{1}{\mu(B(\xi,2r))}\int_{B(\xi,2r)}|f|\, d\nu
\le C_D^2\, \M\, (2r)^{\alpha}. \qedhere
\]
\end{proof}

\begin{remark}
If $f\in M^{1,-(\alpha+\Theta)}(\partial\Om,\nu)$ and $\alpha+\Theta>0$ then $f=0$, and if $\alpha+\Theta=0$ then $f$ is 
in $L^\infty(\partial\Om,\nu)$. In both of those cases we have H\"older continuity of $u$ from the prior
work~\cite{CGKS, CKKSS}. Hence the interesting part is
the case when  $\alpha+\Theta<0$. 
\end{remark}

The following proposition is a variant of a result of Campanato~\cite{Cam} in the Euclidean setting, and of
Da Prato~\cite{DaP} for Ahlfors regular distances in $\R^n$.  We include the proof here to keep our discussion self-contained.

\begin{prop} \label{prop:Luca-Campanato}
Let $(Y,d_Y,\mu_Y)$ be a doubling metric measure space supporting an $s$-Poincar\'e inequality for some $1\leq s <\infty$. 		
Suppose that $B$ is a ball in $Y$, $u\in N^{1,s}(B)$,
and suppose that $|\nabla u|\in M^{s,\lambda}(B)$ for some
$0<\lambda<1$. Then $u$ is locally $(1-\lambda)$-H\"older continuous on $\tfrac12 B$; that is, there exists some 
$C_*\ge 1$, depending only on the structural constants and $[|\nabla u|]_{M^{s,\lambda}}$,
such that whenever
$x,y\in Y$ with $d_Y(x,y)<R_0/5$, we have $|u(x)-u(y)|\le C_*\, d_Y(x,y)^{1-\lambda}$.
\end{prop}

In the above proposition, $R_0$ is the scale limit in the Morrey space definition, Definition~\ref{def:Morrey}.

\begin{proof}
We first prove the above claimed H\"older estimate for $x,y\in \tfrac12 B$ that are Lebesgue points of $u$; recall
from~\cite{KinLat} (or~\cite[Theorem~9.2.8]{HKST})
that $p$-capacity almost every point in $Y$ is such a 
point. 

Let $x,y\in \tfrac12B$ be Lebesgue points of $u$
such that $d_Y(x,y)<R_0/5$, and set $r=d_Y(x,y)$. For positive integers $i$ we set
$B_i=B(x,2^{1-i}r)$ and $B_{-i}=B(y,2^{1-i}r)$. We also set $B_0=B(x,2r)$.
Then $\lim\limits_{i\to\infty} u_{B_i}=u(x)$ and $\lim\limits_{i\to\infty}u_{B_{-i}}=u(y)$; it follows that
\[
|u(y)-u(x)|\le \sum_{i\in\Z} |u_{B_i}-u_{B_{i+1}}|.
\]
Since $B_{i+1}\subset 4B_i$, by the doubling property of $\mu_Y$ and the $s$-Poincar\'e inequality, we have that
\[
|u(y)-u(x)|\le C\, \sum_{i\in\Z}\jint_{4B_i}|u-u_{4B_i}|\, d\mu_Y
  \le C\, \sum_{i\in \Z} 2^{-|i|}r\, \left(\jint_{4B_i}|\nabla u|^s\, d\mu_Y\right)^{1/s}.
\]
Applying the assumption that $|\nabla u|\in M^{s,\lambda}(Y,\mu_Y)$, we now have
\begin{align*}
|u(y)-u(x)|&\le C\, \sum_{i\in\Z}\, (2^{-|i|}\, r)^{1-\lambda}\, \left((2^{-|i|}r)^{\lambda s}\, \jint_{4B_i}|\nabla u|^s\, d\mu_Y\right)^{1/s}\\
&\le 2C\, [|\nabla u|]_{M^{s,\lambda}}\ r^{1-\lambda}\, \sum_{i\in \Z} 2^{-|i|(1-\lambda)}.
\end{align*}
As $r=d_Y(x,y)$, the claim follows with 
\[
C_*=2C\, [|\nabla u|]_{M^{s,\lambda}}\, \sum_{i\in \Z} 2^{-|i|(1-\lambda)}.
\]
\end{proof}

In our application of this proposition, we will have a ball $B$ in the metric space $Y=\overline{\Om}$
and $s=p>1$. Note that when $u\in N^{1,p}(B\cap\Om)=N^{1,p}(B)$,
by~\cite[Proposition~3.11]{GKS}, we not only have that $p$-capacity almost every point in $B$
is a Lebesgue point, but also that $\nu$-almost 
every point in $B\cap\partial\Om$ is such a point.

\section{Proof of Morrey type estimate for the Neumann problem}\label{proofs}

In this section, we prove Theorem \ref{Morrey type estimate} and Theorem \ref{new regularity}.

\begin{lem}\label{lem:Estimate1}
Under the hypotheses of Theorem~\ref{Morrey type estimate}, 
there is a constant $C>0$ such that for all $0<r<R$, $x_0\in B(z_0,R)$, and $0<\eps<1$ we have
\begin{align}\label{eq:preGiaq}
\int_{B(x_0,r)}|\nabla u|^p\, d\mu  \le &
\,\frac{2C\, \mathcal{M}^{p^\prime}}{p^\prime\, \eps^{1/(p-1)}}\,  \mu(B(x_0,R))\, R^{\kappa p^\prime}
+\notag\\ & + 
\left(\frac{2\eps}{p}+2C\, \left(\frac{r}{R}\right)^{\tau p-p}\,\frac{\mu(B(x_0,r))}{\mu(B(x_0,R))}\, \right)\, \int_{B(x_0,R)}|\nabla u|^p\, d\mu.
\end{align}
\end{lem}

\begin{proof}
Let $v$ be the solution to the Dirichlet problem (for $p$-harmonicity) on $B(x_0,R)\subset \overline \Om$ with boundary data $u$. Then
\begin{align*}
\int_{B(x_0,r)}|\nabla u|^p\, d\mu
&=\int_{B(x_0,r)}\left(|\nabla u|^{p-2}\nabla u-|\nabla v|^{p-2}\nabla v\right)\,\cdot\,(\nabla u-\nabla v)\, d\mu \,+\\
     &\qquad\quad+\int_{B(x_0,r)}|\nabla u|^{p-2}\nabla u\, \cdot\, \nabla v\, d\mu
        +\int_{B(x_0,r)}|\nabla v|^{p-2}\nabla v\,\cdot\,(\nabla u-\nabla v)\, d\mu\\
 &\le \int_{B(x_0,R)}\left(|\nabla u|^{p-2}\nabla u-|\nabla v|^{p-2}\nabla v\right)\,\cdot\,(\nabla u-\nabla v)\, d\mu\, +\\
      &\qquad\quad+\int_{B(x_0,r)}|\nabla u|^{p-2}\nabla u\, \cdot\, \nabla v\, d\mu
        +\int_{B(x_0,r)}|\nabla v|^{p-2}\nabla v\,\cdot\,(\nabla u-\nabla v)\, d\mu\\
  &=\int_{B(x_0,R)}|\nabla u|^{p-2}\nabla u\,\cdot\,(\nabla u-\nabla v)\, d\mu\,+\\
      &\qquad\quad+\int_{B(x_0,r)}|\nabla u|^{p-2}\nabla u\, \cdot\, \nabla v\, d\mu
        +\int_{B(x_0,r)}|\nabla v|^{p-2}\nabla v\,\cdot\,(\nabla u-\nabla v)\, d\mu\\
  &=\int_{B(x_0,R)}(u-v)\, \, d\barnu+\int_{B(x_0,r)}|\nabla u|^{p-2}\nabla u\, \cdot\, \nabla v\, d\mu\,+\\
       &\qquad\qquad\qquad\qquad\qquad\quad +\int_{B(x_0,r)}|\nabla v|^{p-2}\nabla v\,\cdot\,(\nabla u-\nabla v)\, d\mu\\
  &\le \int_{B(x_0,R)}(u-v)\,  d\barnu+\int_{B(x_0,r)}|\nabla u|^{p-1}\, |\nabla v|\, d\mu\,+\\
     &\qquad\qquad\qquad\qquad\qquad\quad +\int_{B(x_0,r)}|\nabla v|^{p-1}\, |\nabla u|\, d\mu-\int_{B(x_0,r)}|\nabla v|^p\, d\mu.
\end{align*}
In the above, the first inequality was due to the fact that the integrand in the first term on the 
right-hand side of the first line is non-negative;
the next equality was obtained using the fact that $v$ is $p$-harmonic in $B(x_0,R)$ together with the fact that $u-v=0$ in 
$X\setminus B(x_0,R)$. The equality after that was obtained using the fact that $u$ 
satisfies~\eqref{eq:Neumann-bar} together with
the fact that $u-v=0$ in $X\setminus B(x_0,R)$. It follows that
\begin{equation}\label{eq:est1}
\int_{B(x_0,r)}|\nabla u|^p\, d\mu\le 
\int_{B(x_0,R)}(u-v)\,  d\barnu+\int_{B(x_0,r)}\left[|\nabla u|^{p-1}\, |\nabla v|\, +\, |\nabla v|^{p-1}\, |\nabla u|\right]\, d\mu.
\end{equation}

We set
\begin{align*}
I_1&:=\int_{B(x_0,R)}(u-v)\, d\barnu,\\
I_2&:=\int_{B(x_0,r)}|\nabla v|^{p-1}\, |\nabla u|\, d\mu+\int_{B(x_0,r)}|\nabla u|^{p-1}|\nabla v|\, d\mu,
\end{align*}
and note that $u-v\in N^{1,p}_0(B(x_0,R))$. 
Thus, to estimate $I_1$ we utilize the
Adams inequality from Lemma~\ref{lem:Adams t}.
Applying H\"older's inequality and subsequently by~\eqref{eq:nu-f-growth}, and then finally by the Adams inequality, 
we obtain 
\begin{align}\label{bounded-case}
I_1&\le \int_{B(x_0,R)}|u-v|\,  d|\barnu| \notag\\ 
    &\le |\barnu|(B(x_0,R))^{1/p^{*\prime}}\left(\int_{B(x_0,R)}|u-v|^{p^*}\, d|\barnu|\right)^{1/p^{*}}\, \notag \\
    &\le C|\barnu|(B(x_0,R))^{1/p^{*\prime}}  \left(\mathcal{M}\mu(B(x_0,R)) R^{\alpha}\right)^{1/p^{*}} R\left(\jint_{B(x_0,R)}|\nabla (u-v)|^p\, d\mu\right)^{1/p} \notag \\
& \le C \mathcal{M}^{1/p^{*}}|\barnu|(B(x_0,R))^{1/p^{*\prime}} \mu(B(x_0,R))^{1/p^{*}-1/p}\, R^{\alpha/p^{*}+1} \left(\int_{B(x_0,R)}|\nabla (u-v)|^p\, d\mu\right)^{1/p} \notag\\
&\le C\mathcal{M} \mu(B(x_0,R))^{1/p'}\, R^{\alpha +1}    \left(\int_{B(x_0,R)}|\nabla (u-v)|^p\, d\mu\right)^{1/p}.
\end{align}

As $v$ is $p$-harmonic in $B(x_0,R)$ and $v-u\in N^{1,p}_0(B(x_0,R))$, we have that
\[
\int_{B(x_0,R)}|\nabla (u-v)|^p\, d\mu\le 2^{p-1}\int_{B(x_0,R)}\left(|\nabla u|^p+|\nabla v|^p\right)\, d\mu
    \le 2^p\, \int_{B(x_0,R)}|\nabla u|^p\, d\mu.
\]
Thus we obtain, for $\eps>0$,
\[
I_1\le \left[\frac{C}{\eps^{1/p}}\, \mathcal{M}\, \mu(B(x_0,R)^{1/p'}\, R^{\alpha+1}\right]\, 
\left(\eps\, \int_{B(x_0,R)}|\nabla u|^p\, d\mu\right)^{1/p}.
\]
For ease of notation we set
\begin{equation}\label{eq:kappa}
\kappa=
\alpha +1.
\end{equation}
Using Young's inequality, we now get
\begin{align}\label{eq:I1-est}
I_1&\le \frac{\eps}{p}\, \int_{B(x_0,R)}\, |\nabla u|^p\, d\mu+\frac{1}{p^\prime}\, 
\left[\frac{C}{\eps^{1/p}}\, \mathcal{M}\, \mu(B(x_0,R)^{1-\tfrac{1}{p}}\, R^{\alpha +1}\right]^{p^\prime}\notag\\
 &=\frac{\eps}{p}\, \int_{B(x_0,R)}\, |\nabla u|^p\, d\mu+\frac{C\, \mathcal{M}^{p^\prime}}{ \eps^{1/(p-1)}}\, \mu(B(x_0,R))\, R^{(\alpha+1) p^\prime}.
\end{align}
Now we turn our attention to estimating $I_2$. By Young's inequality applied to each of the two terms comprising $I_2$, we know that
\begin{align*}
I_2&\le \int_{B(x_0,r)}\left[\frac{|\nabla u|^p}{2p^\prime}\, +\, \frac{2^{p/p^\prime}|\nabla v|^p}{p}
+\frac{2^{p^\prime/p}|\nabla v|^p}{p^\prime}\, +\, \frac{|\nabla u|^p}{2p}\right]\, d\mu\\
  &=\frac{1}{2}\, \int_{B(x_0,r)}|\nabla u|^p\, d\mu+\left(\frac{2^{p-1}}{p}+\frac{2^{p^\prime-1}}{p^\prime}\right)\, \int_{B(x_0,r)}|\nabla v|^p\, d\mu.
\end{align*}
From Lemma~\ref{lem:harm-energy-decay}, we have 
\[
\int_{B(x_0,r)}|\nabla v|^p\, d\mu\le C_0\, \left(\frac{r}{R}\right)^{\tau p-p}\, \frac{\mu(B(x_0,r))}{\mu(B(x_0,R))}\, \int_{B(x_0,R)}|\nabla v|^p\, d\mu.
\]
Thus, we obtain the estimate 
\begin{align}\label{eq:I2-est}
I_2&\le \frac{1}{2}\, \int_{B(x_0,r)}|\nabla u|^p\, d\mu+C\, \left(\frac{r}{R}\right)^{\tau p-p}\, \frac{\mu(B(x_0,r))}{\mu(B(x_0,R))}\, \int_{B(x_0,R)}|\nabla v|^p\, d\mu\notag\\
&\le \frac{1}{2}\, \int_{B(x_0,r)}|\nabla u|^p\, d\mu+C\, \left(\frac{r}{R}\right)^{\tau p-p}\, \frac{\mu(B(x_0,r))}{\mu(B(x_0,R))}\, \int_{B(x_0,R)}|\nabla u|^p\, d\mu.
\end{align}
Combining~\eqref{eq:I1-est} and~\eqref{eq:I2-est} together with~\eqref{eq:est1}, we obtain the following inequality:
\begin{align*}
\int_{B(x_0,r)}|\nabla u|^p\, d\mu\le 
\frac{\eps}{p}\, \int_{B(x_0,R)}\, |\nabla u|^p\, d\mu+\frac{C\, \mathcal{M}^{p^\prime}}{p^\prime\, \eps^{1/(p-1)}}\, \mu(B(x_0,R))\, R^{\kappa p^\prime}
+\notag\\
+\frac{1}{2}\, \int_{B(x_0,r)}|\nabla u|^p\, d\mu+
C\, \left(\frac{r}{R}\right)^{\tau p-p}\, \frac{\mu(B(x_0,r))}{\mu(B(x_0,R))}\, \int_{B(x_0,R)}|\nabla u|^p\, d\mu.
\end{align*}
Simplifying, we obtain
\begin{align*}
\int_{B(x_0,r)}|\nabla u|^p\, d\mu  \le &
\,\frac{2C\, \mathcal{M}^{p^\prime}}{p^\prime\, \eps^{1/(p-1)}}\,  \mu(B(x_0,R))\, R^{\kappa p^\prime}
+\notag\\ & + 
\left(\frac{2\eps}{p}+2C\, \left(\frac{r}{R}\right)^{\tau p-p}\,\frac{\mu(B(x_0,r))}{\mu(B(x_0,R))}\, \right)\, \int_{B(x_0,R)}|\nabla u|^p\, d\mu
\end{align*}
as desired.
\end{proof}
For subsequent use, we set
$\delta:=\tau p-p-\kappa p'$ and note here that $\delta>0$.

At this point we recall that $z_0\in \overline \Omega$ is fixed and $0<r<R$. If we  set 
 $\phi(r)=\int_{B(x,r)}|\nabla u|^p\, d\mu$ and $\omega(r)=\mu(B(x,r))$, then the previous estimate reads as 
 
 \[
\phi(r)\le A_1\, \left[\frac{\omega(r)}{\omega(R)}\, \left(\frac{r}{R}\right)^{\tau p-p}+\eps\right]\, \phi(R)+A_2(\varepsilon)\, \omega(R)\, R^{\kappa p'}
\]
for every $0<r<R$ and for every $0<\varepsilon<1$. In order to conclude the proof
of Theorem~\ref{Morrey type estimate}, we invoke Lemma~\ref{lem:Gia-Mak} below.
The following lemma is a version of~\cite[Lemma~2.1, page~86]{Gia} 
for measures that are not necessarily   Ahlfors regular. Our  version is essentially the same 
as~\cite[Lemma 2.7]{Mak-Holder} with the choice of $\delta=\tau p-p-\kappa p'$ and 
$\beta =\kappa p'$, but as we use it with different parameters, we included the proof here for the sake of completeness.

\begin{lem}\label{lem:Gia-Mak}
Let $\phi$ and $\omega$ be two non-negative and  non-decreasing functions on an interval $(0,R_0]$ and assume that 
$\alpha<\tau p-p-\tau$ and that there are positive
constants $C_1$ and $s$ such that for all $0<\lambda\le 1$ and $0<r<R$ we have
\begin{equation}\label{eqn:omega}
\frac{\omega(\lambda r)}{\omega(r)}\ge C_1\, \lambda^s,
\end{equation}
and that there is a constant $A_{1}> 1$ and a function $A_{2}:(0,\infty)\rightarrow(0,\infty)$ such that
\[
\phi(r)\le A_1\, \left[\frac{\omega(r)}{\omega(R)}\, \left(\frac{r}{R}\right)^{\tau p-p}+\eps\right]\, \phi(R)+A_2(\varepsilon)\, \omega(R)\, R^{\kappa p'}
\]
for every $0<r<R$ and for every $0<\varepsilon<1$.
Then there are positive constants $C$ and $\varepsilon_{0}$ that depend only on $p, C_1, s, A_1, A_2, \kappa=\alpha+1$,
and $\delta=\tau p-p-\kappa p'$ so 
that for every $0<r<R$ we have
\begin{equation}\label{eq:Morrey-Alternate}
\phi(r)\le C\, \left[\frac{\omega(r)}{\omega(R)}\, \left(\frac{r}{R}\right)^{\kappa p'}\, \phi(R)+A_2(\varepsilon_{0})\, \omega(r)\, r^{\kappa p'}\right].
\end{equation}
\end{lem}

\begin{remark} We note explicitly that the quantity $s$ from \eqref{eqn:omega} only appears in the constants $C,\eps_0$ and not in the decay exponent itself.

Note also that we need $\tau p-p-\kappa p'>0$ in the proof of the above lemma. 
Recalling that $\kappa=\alpha+1$, this is equivalent to the condition
$\alpha<\tau p-p-\tau$ stated in the lemma.
\end{remark}

\begin{proof}
If $0<\lambda<1$, we have
\[
\phi(\lambda R)\le A_{1} \lambda^{\tau p-p}\left[   \frac{\omega(\lambda R)}{\omega(R)} 
    + \varepsilon \lambda^{p-\tau p}  \right]\phi(R)+A_{2}(\varepsilon)\omega(R)R^{\kappa p'}.
\]
Let us choose $0<\lambda_{0}<1$ so that $(2A_{1})^{2}\lambda_{0}^{\tau p-p-\kappa p'}=1$ and 
we choose $\eps_0>0$ such that
$\varepsilon_{0}=C_{1}\lambda_{0}^{{s+\tau p-p}}$. Then it follows by \eqref{eqn:omega} that
\begin{equation}\label{eq:lambda0}
\varepsilon_{0}\lambda_0^{p-\tau p}=C_1\lambda_0^s\le \frac{\omega(\lambda_{0}R)}{\omega(R)}.
\end{equation}
Consequently, we have
\[
\begin{split}
\phi(\lambda_0 R) &\le 2 A_{1}\lambda_{0}^{\tau p -p}\,\frac{\omega(\lambda_{0}R)}{\omega(R)}\,\phi(R)
  + A_2(\varepsilon_{0})\omega(R)R^{\kappa p'}\\
&\le\lambda_{0}^{(\tau p -p+\kappa p')/2}\,\frac{\omega(\lambda_{0}R)}{\omega(R)}\,\phi(R)
  + \frac{A_2 (\varepsilon_{0})}{C_{1}}\, \lambda_{0}^{-s}\omega(\lambda_0 R)R^{\kappa p'}.
\end{split}
\]
To simplify notation, recall that we set $\delta=\tau p-p-\kappa p'$, and so $\tau p-p+\kappa p'=\delta+2\kappa p'$.
In obtaining the second term in 
the second inequality, we used~\eqref{eq:lambda0}. 
By iterating this estimate, 
we obtain that for all positive integers, we have 
\[
\begin{split}
\phi(\lambda_{0}^{k}R) & \le
\lambda_{0}^{\kappa p'+\delta/2}\,\frac{\omega(\lambda_{0}^{k}R)}{\omega(\lambda_{0}^{k-1}R)}\phi(\lambda_{0}^{k-1}R)
+ \frac{A_2 (\varepsilon_{0})}{C_{1}}\,\lambda_{0}^{-s}\omega(\lambda_{0}^{k} R)\left(\lambda_{0}^{k-1}R\right)^{\kappa p'}\\
 & \le
 \lambda_{0}^{k\,(\kappa p'+\delta/2)}\,\frac{\omega(\lambda_{0}^{k}R)}{\omega(R)}\phi(R)
+\frac{ A_2 (\varepsilon_{0})}{C_{1}}\,\lambda_{0}^{-s}R^{\kappa p'} \lambda_{0}^{(k-1)\kappa p'}\omega(\lambda_{0}^{k} R)\sum_{j=0}^{k-1}\left(\lambda_{0}^{\delta/2}\right)^{j}\\
& \le
 \lambda_{0}^{k\, (\kappa p'+\delta/2)}\,\frac{\omega(\lambda_{0}^{k}R)}{\omega(R)}\phi(R)
 + \frac{A_2 (\varepsilon_{0})\, R^{\kappa p'}\, \lambda_{0}^{(k-1)\kappa p'-s}\, \omega(\lambda_{0}^{k} R)}{C_1\,(1-\lambda_{0}^{\delta/2})}.
\end{split}
\]
Notice that the series corresponding to the sum above converges exactly when $\tau p-p-\kappa p'=\delta>0$. 
Now we choose the unique positive integer $k$ so that $\lambda_{0}^{k+1}R<r<\lambda_{0}^{k}R$. Then 
\begin{align*}
\phi(r)&\le \phi(\lambda_0^kR)\\
 &\le  \lambda_{0}^{k\, (\kappa p'+\delta/2)}\frac{\omega(\lambda_{0}^{k}R)}{\omega(R)}\phi(R)
 + \frac{A_2 (\varepsilon_{0})\, R^{\kappa p'}\, \lambda_{0}^{(k-1)\kappa p'-s}\, \omega(\lambda_{0}^{k} R)}{C_1\,(1-\lambda_{0}^{\delta/2})}\\
 &\le
 \left(\frac{r}{R}\right)^{\kappa p'+\delta/2}\, 
     \frac{\omega(\lambda_0^{k+1}R)}{\omega(R)}\, \frac{1}{C_1\, \lambda_0^s}\, \phi(R)
   +\frac{A_2 (\varepsilon_{0})\, R^{\kappa p'}\, \lambda_{0}^{(k-1)\kappa p'-s}\, \omega(\lambda_{0}^{k+1} R)}{C_1\,(1-\lambda_{0}^{\delta/2})}
   \frac{1}{C_1\, \lambda_0^s}\\
  &\le \frac{1}{C_1\, \lambda_0^{s}}\, \left(\frac{r}{R}\right)^{\kappa p'+\delta/2}\, \frac{\omega(r)}{\omega(R)}\, \phi(R)
    +\frac{A_2 (\varepsilon_{0})\, R^{\kappa p'}\, \lambda_{0}^{k\, \kappa p'}}{C_1^2\, \lambda_0^{\kappa p'+2s}\,(1-\lambda_{0}^{\delta/2})}\,  \omega(r)\\
  &\le \frac{1}{C_1\, \lambda_0^{s}}\, \left(\frac{r}{R}\right)^{\kappa p'+\delta/2}\, \frac{\omega(r)}{\omega(R)}\, \phi(R)
    +\frac{A_2 (\varepsilon_{0})}{C_1^2\, \lambda_0^{\kappa p'+2s}\,(1-\lambda_{0}^{\delta/2})}\,  \omega(r)\, r^{\kappa p'}.
\end{align*}
In the last step, we used the fact that $\kappa=\alpha+1<0$.
We can set
\[
C=\max\bigg\lbrace\frac{1}{C_1\, \lambda_0^{s+\kappa p' +\delta/2}},\ \frac{1}{C_1^2\, (1-\lambda_0^{\delta/2})\, \lambda_0^{\kappa p'+2s}}\bigg\rbrace.\qedhere
\]
\end{proof}

\begin{proof}[Proof of Theorem~\ref{Morrey type estimate}]
{We note that
$\omega(r):=\mu(B(x_0,r))$ satisfies the  hypothesis \eqref{eqn:omega}
 with the choice of $s=Q_\mu$,
where $Q_\mu$ is the lower mass bound exponent associated with $\mu$ as in~\eqref{eq:lower-mass-exp}.}
We apply Lemma~\ref{lem:Gia-Mak} to the 
inequality obtained in Lemma~\ref{lem:Estimate1} with the choice of $\omega(r)=\mu(B(x_0,r))$ and
$\phi(r):=\int_{B(x_0,r)}|\nabla u|^p\, d\mu$ to obtain that
when $0<r<R$,
\[
\int_{B(x_0,r)}|\nabla u|^p\, d\mu
\le C\, \left[\, \left(\frac{r}{R}\right)^{\kappa p'}\, \frac{\mu(B(x_0,r))}{\mu(B(x_0,R))}\, \int_{B(x_0,R)}|\nabla u|^p\, d\mu
+A_2(\eps_0)\, \mu(B(x_0,r))\, r^{\kappa p'}\right].
\]
Letting
\[
C_1:= C\, \left[ R^{-\kappa p'}\, \jint_{B(x_0,R)}|\nabla u|^p\, d\mu+A_2(\eps_0)\right],
\]
it follows that
\[
\jint_{B(x_0,r)}|\nabla u|^p\, d\mu\le C_1\, r^{\kappa p'}.
\]
Since $\kappa p'= p\, \tfrac{1+\alpha}{p-1}$, it follows that $|\nabla u|\in M^{p,\tfrac{1+\alpha}{1-p}}(B(z_0,R))$
with the Morrey scale $R_0=R$. This completes the proof of the first part of Theorem~\ref{Morrey type estimate}.
The last part now follows from an application of Proposition~\ref{prop:Luca-Campanato}.
\end{proof}

Next we turn to the proof of Theorem~\ref{new regularity}.

\begin{proof}[Proof of Theorem~\ref{new regularity}]
The first statement follows at once from Theorem \ref{Morrey type estimate} applied to the measure 
$d \bar \nu=f d\nu$. In fact, since $f\in M^{1,-(\alpha+\Theta)}(\partial\Om,\nu)$ then in view of 
Lemma~\ref{membership-Morrey}, we know that
\eqref{eq:nu-f-growth} holds, and thus 
Theorem~\ref{Morrey type estimate}  yields the desired conclusion.

To prove the second statement, without loss of generality we may assume  $f\ge 0$ on $B_{4R}$.
As $f\ge 0$ on
$B_{4R}$ and $u$ is the solution to the Neumann problem with measure data $f\, d\nu$, therefore
$u$ is $p$-superharmonic on $B_{4R}$
and hence by~\cite[Lemma~4.8]{BMcSh} we know that
for $x\in B_R$ and $0<r\le R$,
\[
\frac{r^p}{\mu(B(x,r))}\, \int_{B(x,r)}|f|\, d\nu\le C\, \left(\sup_{B(x,2r)}u\, -\inf_{B(x,2r)}u\right)^{p-1}.
\]
So if $u$ is $\lambda$-H\"older continuous on $B_{2R}$,
we must have 
\[
\frac{1}{\mu(B(x,r))}\int_{B(x,r)}|f|\, d\nu\le C\, r^{-p}\, r^{\lambda(p-1)}=C\, r^{\lambda (p-1)-p}.
\]
Thus, $f$ satisfies the decay condition of~\eqref{eq:Mor-f} with $\alpha=\lambda(p-1)-p$. 
\end{proof}

\section{An improved H\"older continuity for solutions to fractional $p$-Laplacian-type equations
}

In this section, we apply the results from the previous section to prove sharp H\"older continuity for solutions of PDE involving fractional  powers of $p$-Laplacian operators on a compact doubling metric measure space $(Z,d_Z,\nu)$. We apply the discussion of the previous sections to
the situation where $\overline\nu$ is given by $d\overline\nu=f\, d\nu$ where $f\in L^{p'}(\partial\Om, \nu)$ represents the right hand side of the nonlocal PDE. 

 For $0<\theta<1$ and $1<p<\infty$ we will consider the following
Besov energy:
\begin{equation}\label{besov energy}
\Vert u\Vert_{\theta,p}^p:=
\int_{Z}\int_{Z}\frac{|u(y)-u(x)|^p}{d_Z(x,y)^{\theta p}\, \nu(B(x,d_z(x,y)))}\, d\nu(y)\, d\nu(x),
\end{equation}
and set $B^\theta_{p,p}(Z)$ to be the space of all $L^p$-functions for which this energy is finite.

We invoke the  uniformization result
 in \cite{BBS}: {\it 
given parameters $1<p<\infty$ and $0<\theta<1$,
every  doubling metric measure space $(Z,d_Z,\nu)$ arises as the boundary of a uniform domain $(\Om,d_X)$ that is equipped
with a measure $\mu$ so that the metric measure space $X=\overline{\Om}=\Om\cup Z$, together with
$Z=\partial\Om$, satisfies conditions~(H0),~(H1) and~(H2), with  $\Theta=p(1-\theta)$. The metric on $\partial \Omega$ 
is induced by the metric on $\Omega$ and while it may not coincide with the original metric $d_Z$ on $Z$, it is in the same bi-Lipschitz class. }

After choosing a Cheeger differential structure $\nabla$ on $\Om$, 
 we proved in \cite{CKKSS,CGKS} that:
\begin{enumerate}
\item For each  function $u\in B^\theta_{p,p}(Z)$, one can find $\widehat{u}$, the unique 
Cheeger $p$-harmonic function in $N^{1,p}(\Om)$ such that
$\widehat{u}$ has 
trace $Tr(\widehat{u})=u$ $\nu$-almost everywhere on $Z$.

\item The Besov norm $\Vert u\Vert_{\theta,p}$ is equivalent to $p$-energy of the extension $\widehat{u}$ of $u$, i.e.  $\int_\Om|\nabla \widehat{u}|^p\, d\mu\approx \Vert u\Vert_{\theta,p}^p$. We then set
\begin{equation}\label{eq:ET-go-home}
\mathcal{E}_T(u,v):=\int_\Om|\nabla\widehat{u}|^{p-2}\nabla\widehat{u}\cdot\nabla \widehat{v}\, d\mu.
\end{equation}
\end{enumerate}

\begin{definition} A function $u\in B^\theta_{p,p}(Z)$ is in the domain of the fractional $p$-Laplacian operator
$(-\Delta_p)^{\theta}$ if there is a function $f\in L^{p'}(Z, \nu)$ such that the integral identity
\[
\mathcal{E}_T(u,\pip)=\int_Z\pip\, f\, d\nu
\]
holds for every  $\pip\in B^\theta_{p,p}(Z)$. 
We then denote
\[
(-\Delta_p)^{\theta} u=f\in L^{p'}(Z,  \nu).
\]
\end{definition}

As a consequence of Theorem~\ref{new regularity}, we obtain the following regularity result for solutions of the fractional $p$-Laplacian.
In what follows, $\delta_F$ is as in Remark~\ref{rem:fat}, and $\tau$ is the interior H\"older regularity assumption, both of which
are determined by the constants for the structural conditions associated with the uniformization $\Om$ of the hyperbolic filling of $Z$
as described at the beginning of the present section.
 
\begin{thm}\label{nonlocal} 
In the hypotheses above, let $u\in B^\theta_{p,p}(Z)$ be a solution to the equation 
\begin{equation}\label{nh}
(-\Delta_p)^\theta u =f,
\end{equation}
for $\theta\in (0,1)$ and $1<p<\infty$. Fix $\xi\in Z$ and $R_0>0$.
\begin{enumerate}
\item If $f\in L^{p'}(Z, \nu)\cap M^{1,-(\alpha+\Theta)}( B(\xi,4R_0),\nu)$, then 
$u$ is $(1-\lambda)$-H\"older continuous on $B(\xi,R_0))$ with $\lambda=\frac{1+\alpha}{1-p} \in (0,1)$.

\item If  
$f$ does not change sign in the ball $B(\xi,4R_0)$ and $u$ is $(1-\lambda)$-H\"older continuous on $B(\xi,2R_0)$ for
some $0<\lambda<1$, then with $\lambda_0\in (0,1)$ given by the equation $1-\lambda_0=\min\{1-\lambda, \tau,\delta_F\}$,
necessarily $f$ belongs to the Morrey space 
$M^{1,-(\alpha+\Theta)}(B(\xi,R_0),\nu)$
with $\alpha=(1-\lambda_0)(p-1)-p$.
\end{enumerate}
\end{thm}

\begin{proof} 
In the following, we extend the measure $\nu$ to all of $X$ by setting $\nu(\Om)=0$.
We first address the second claim, and note that the proof is very similar to the argument for the second statement of 
Theorem~\ref{new regularity} .
Without loss of generality we  may assume that $f\ge 0$ on $B(\xi,4R_0)$.
From~\cite[Lemma~4.8]{BMcSh},  it follows that as $f\ge 0$ on
$B(\xi,4R_0)$ and $\widehat{u}$ is the solution to the Neumann problem with measure data $f\, d\nu$ in $\Om$, then 
$\widehat{u}$ is $p$-superharmonic on $B_{\overline\Om}(\xi,4R_0)=\{y\in\overline\Om:d_X(\xi,y)<4R_0\}$,
and consequently for every $x\in B_{\overline\Om}(\xi,R_0)$ and $0<r\le R_0$, we have
\[
\frac{r^p}{\mu(B_{\overline\Om}(x,r))}\, \int_{B_{\overline\Om}(x,r)}|f|\, d\nu\le C\, \left(\sup_{B_{\overline\Om}(x,2r)}\widehat{u}\, -\inf_{B_{\overline\Om}(x,2r)}\widehat{u}\right)^{p-1}.
\]
Since $u=Tr(\widehat{u})$ is $(1-\lambda)$-H\"older continuous on $B(\xi,2R_0)$ and $1-\lambda<\tau$, then 
{from Lemma~\ref{lem:Loewner} and the subsequent Remark~\ref{rem:fat},}
we must have that $\widehat{u}$ is also ${(1-\lambda_0)}$-H\"older continuous in $B_{\overline\Om}(\xi,2R_0)$ 
and so
\[
\frac{1}{\mu(B_{\overline{\Om}}(x,r))}\int_{B_{\overline{\Om}}(x,r)}|f|\, d\nu\le C\, r^{-p}\, r^{(1-\lambda)(p-1)}=C\, r^{(1-\lambda)(p-1)-p}.
\]
Thus $f$ satisfies the decay condition of~\eqref{eq:Mor-f} with $\alpha={(1-\lambda_0)}(p-1)-p$. 

Next we prove the first claim. Consider $u\in B^\theta_{p,p}(Z)$ solution of $(-\Delta_p)^{\theta}u =f$ and let $\widehat{u}$ be as in \eqref{eq:ET-go-home}, satisfying
\begin{equation}\label{eq:ET-go-home-1}
\int_{\overline\Om}|\nabla\widehat{u}|^{p-2}\nabla\widehat{u}\cdot\nabla \widehat{v}\, d\mu= \int_{\partial \Omega} \pip\, f\, d\nu
\end{equation}
for every  $\pip\in B^\theta_{p,p}(Z)$.  By Theorem \ref{new regularity} 
we have that 
 $\widehat u$ is H\"older continuous in $\overline \Omega$ with H\"older exponent $\frac{p+\alpha}{p-1}\in (0,1)$. Since $u$ is the trace of $\widehat u$ on $\partial \Omega$, then it shares the same H\"older regularity, thus completing the argument.
\end{proof}

In terms of the hypotheses needed from $f\in L^q(Z,\nu)$  to guarantee membership in the appropriate Morrey space, we note that by choosing $q>0$ such that
$1-\lambda=\tfrac{q(p-\Theta)-Q_\mu+\Theta}{q(p-1)}$, we get that the decay index is
$\alpha=-\tfrac{Q_\mu+(q-1)\Theta}{q}$, as desired.

\section{Comparison with existing literature in the Euclidean setting}

There is a vast literature concerning the study of the fractional $p$-Laplacian partial differential equations
in the Euclidean setting (see for instance \cite{BT, BLS,DK0,FRRO, GL} and the references therein).  When $p\neq 2$, 
the definition of fractional $p$-Laplacian used in most of these papers is different from ours, as it concerns minimizers of the 
Besov energy~\eqref{besov energy}, while in this paper we follow the approach in \cite{CGKS, CKKSS}, and study minimizers of the equivalent energy~\eqref{eq:ET-go-home}.

The purpose of this section is to illustrate how, despite using these different notions of fractional $p$-Laplacian, the 
sharp H\"older exponents for the regularity of solutions are the same.
We will consider two different non-homogeneous partial differential equations
\begin{equation}\label{PDES}
(-\Delta_p u)_E^s =f \text{ and } (-\Delta_p u)^s =f
\end{equation}
in a bounded Euclidean domain $\Omega\subset \R^n$ with $s\in (0,1)$, each involving a different notion of fractional 
$p$-Laplacian. The operator corresponding to the energy~\eqref{besov energy} is
\[
 (-\Delta_p u)_E^s(x) =2\lim_{\epsilon \to 0^+} \int_{\R^n\setminus B(x,\epsilon)}\frac{|u(y)-u(x)|^{p-2}(u(y)-u(x))}{\nu(B(x,d(x,y))\, d(x,y)^{s p}}\, dy,
 \]
where $\nu$ is the Lebegue measure. Interpreting the first equation in \eqref{PDES} in a weak sense, 
whenever $\varphi$ is a smooth function on $\R^n$ with compact support contained in $\Om$, we are required to have
\[
2\, \lim_{\epsilon\to 0^+}\int_{\R^n}\int_{\R^n\setminus B(x,\epsilon)}\varphi(x)\frac{|u(y)-u(x)|^{p-2}(u(y)-u(x))}{\nu(B(x,d(x,y))\, d(x,y)^{s p}}\, dy
\, dx=\int_\Om f(x)\, \varphi(x)\, dx.
\]

The partial differential equation that arises out of minimizing the energy~\eqref{eq:ET-go-home} is the following: 
For  $(x,y)\in \R^n \times \R^+$, consider weak solutions $u$ of the non-linear Neumann problem
\begin{equation}\label{eqn91}
\begin{cases}
\text{div}\bigg( y^a|\nabla u(x,y)|^{p-2} \nabla u(x,y)\bigg) =0 \text{ for } y>0 \text{ and } x\in \R^n\\
\lim\limits_{y\to 0^+} y^a |\nabla u(x,y)|^{p-2}  \partial_y u (x,y) = f(x) \text{ at }x\in \R^n,
\end{cases}
\end{equation}
where $-1<a<p-1$, and $\text{div}$ and 
$\nabla$ refer to the usual differential structure in the Euclidean domain $\R^n\times\R^+$
endowed with Lebesgue measure. In concordance with our notation, we denote by $\nu$ the $n$-dimensional Lebesgue measure on the boundary $\R^n=\partial(\R^n\times\R^+)$. If $u\in N^{1,p}(\R^n\times \R^+{, \mu})$ is a solution of this 
Neumann problem with weighted Lebesgue measure $d\mu(x,y)=y^ad\nu(x)\, dy$ on $\R^n\times\R^+$,
then its trace $Tu$
on 
the boundary $\R^n=\partial(\R^n\times\R^+)$ 
satisfies the fractional partial differential equation
\[
(-\Delta_p)^{s} Tu=f,
\]
with $s=\frac{p-a-1}{p}$.
This follows from \cite{CGKS} and from  the fact that in this setting
we can choose as the uniform domain $\Omega$ the space $(\R^n\times \R^+, d_{Eucl}, y^a d\nu\,d y)$, that is the upper half space endowed with the weighted measure $y^a d\nu\, d y$. 

\begin{remark}\label{remark:CSexample} 
For $p=2$ and $0<s<1$, when the restriction of the Lebesgue measure to $\Omega$ is doubling and satisfies a 
$2$-Poincar\'e inequality, the two versions of the fractional $2$-Laplacian in \eqref{PDES} are equivalent. See~\cite{CS, CGKS} for more details about the 
differences and similarities of the two notions.
Here $a=1-2s$, which means
$-1<a<1$, 
and the codimensionality exponent is $\Theta=2(1-s)=a+1$. 
Note that the optimal lower mass bound exponent for the boundary~\ref{eq:lower-mass-exp-1} is strictly smaller than the lower mass bound exponent for the entire domain~\ref{eq:lower-mass-exp-1} for some values of $a$. Indeed, the measure of a ball of radius $r$ and center $(x,y)\in \R^n\times \R^+$ is roughly $r^{n+1}(y+r)^a$. 
Let us consider  a ball $B_r$ in $\R^n\times\R^+$ with a center at $(x_0,y)$ and a radius $r$. 
When
 $0\le y\le r$, we have that $\mu(B_r)\approx r^{n+1+a}$, and when $y>r$, we have that 
$\mu(B_r)\approx r^{n+1}y^a$.

 Thus for any $y\ge 0, r>0$ we have $ \mu(B_r) \approx r^{n+1} \max\{ r,y \}^a$.
  It follows that if $0<r<R$ and the balls $B_r$ and $B_R$ are centered at the same point $(x_0,y)$, we obtain

\[
\frac{ \mu(B_r) }{ \mu(B_R) } \approx \left( \frac{ r }{ R } \right)^{n+1} \left( \frac{ \max\{y,r\} }{  \max\{y,R\} }\right)^{a}. 
\]
From this estimate we see that the optimal lower mass bound exponent is $Q_{\mu} = n+1+a = n+\Theta$, if $ a \ge 0 $, and $Q_{\mu}=n+1$ if $a<0$. However, for the balls centered at the boundary $\mathbb R^{n}\times \{0\}$, we always have   $\tfrac{ \mu(B_r) }{ \mu(B_R) } \approx \left(\frac{r}{R}\right)^{n+1+a}$ and thus the optimal lower mass bound for the boundary balls is
\begin{equation}\label{Qpartial}
Q^{\partial}_{\mu}=n+1+a=n+\Theta 
\end{equation}
and in particular $Q^{\partial}_{\mu}<Q_{\mu}$ when $s> \tfrac12$, which corresponds to having $a<0$. 
\end{remark}

Notice that for $p=2$, Theorem~\ref{nonlocal} together with Remark~\ref{rem:improvement} give exactly the same exponents as   Caffarelli-Stinga~\cite{CSt},
where the dimension $n$ of the space coincides, via \eqref{Qpartial},  with the natural lower mass bound dimension of $\nu$ here,
which is $Q_\mu^\partial-\Theta$. The fractional power $s$ in~\cite[Theorem~1.2]{CSt} is given 
by $s=1-\Theta/p=1-\Theta/2$.

\medskip

\begin{remark}
	Observe that \eqref{Qpartial} still holds true for all $a>-1$.  
	Given a choice of $p>1$ and $0<s<1$, we choose $\Theta=p(1-s)$, and subsequently we choose $a=\Theta-1$ in the formulation \eqref{eqn91}, yielding $-1<a<p-1$. It follows that with the choice $a=\Theta-1=p(1-s)-1$, \eqref{Qpartial} holds.
\end{remark}

Let us recall from \cite{BLS, GL, BT} the sharp H\"older exponents for solutions of the non-homogeneous fractional $p$-Laplacian partial
differential equation
\[
(-\Delta_p u)^s_E =f
\]
in a bounded Euclidean domain $\Omega\subset \R^n$ with $s\in (0,1)$:
For $f\in L^q(\Omega)$, with  \begin{equation}\label{euclidean range} q>q_0^E:=\frac{n}{ps},\end{equation} one has that 
$u$ is locally $\lambda_E$-H\"older continuous on $\Omega$ with 
\begin{equation}\label{euclidean exponent} 
\lambda_E= \min\Bigg\{ 1, \frac{1}{p-1} \bigg( sp - \frac{n}{q}\bigg) \Bigg\}.
\end{equation}
The authors of those papers also prove that this result is sharp if $sp\le (p-1) + \frac{n}{q}$. This means  that for $q>q_0^E,$
$\frac{spq-n}{q(p-1)}<1$ and $\epsilon>0$ there exists $f\in L^q_{loc}(\R^n)$ with a solution $u$ to 
$(-\Delta_p u)^s_E =f$ that is not $(\lambda_E+\epsilon)$-H\"older continuous.

\medskip

When it comes to the equation 
$$(-\Delta_p u)^s =f$$
in $(\R^n, dx)$ as formulated in \eqref{eqn91}, our main result Theorem~\ref{nonlocal} and Remark~\ref{rem:improvement}, together, yield that when $f\in L^q(Z,d\nu)$ with 
\[
q> q_0:=\frac{Q_\mu^\partial-\Theta}{p-\Theta} = \frac{Q_\mu^\partial-\Theta}{ps},
\] 
the solutions to \eqref{nh} are $\lambda$-H\"older continuous with 
\[
\lambda=\min\left\{1, \frac{1}{p-1}\left((p-\Theta)-\frac{Q_\mu^\partial-\Theta}{q}\right)  
\right\}.
\]

In comparing $q_0, \lambda$ with $q_0^E, \lambda_E$, we note that in view of \eqref{Qpartial} and of the fact that $s=1-\frac{\Theta}{p}$,
 they are the same,  i.e. $q_0=q_0^E$ and $\lambda=\lambda_E$.

\bigskip

	\medskip
	
We now discuss comparisons with the recent manuscript~\cite{BT}, which considers the 
problem in Euclidean domains $W$ with the restriction that the inhomogeneity data $f\in L^\infty(W)$. One of their main results is \cite[Theorem 1.1]{BT}, which states that local weak solutions 
with $f\in L^\infty(W)$ are $\lambda$-H\"older regular with $$\lambda= \min\Bigg\{ 1, \frac{sp}{p-1}\Bigg\}.$$
Using the hypothesis $f\in L^\infty(W)$, in our argument in Remark~\ref{rem:improvement} above, we see that
\[
\frac{1}{\mu(B(x,r))}\, \int_{B(x,r)}|f|\, d\nu\lesssim \Vert f\Vert_{L^\infty(\partial\Om,\nu)}\, r^{-\Theta};
\]
that is, we can choose in this case to have $\alpha=-\Theta$, provided that $\Theta>1$ (which corresponds to our
need to have $\alpha<-1$). This also is in accordance with the computations leading to~\eqref{bounded-case}, 
and so, if $p(1-\tau)+\tau<\Theta<p$, we obtain the same sharp estimate that~\cite[Theorem~1.1]{BT} claims.
Note that their choice of $s$ corresponds to $1-\Theta/p$ in our paper. Thus, in their calculation,
$sp/(p-1)=1$ corresponds to $\Theta=1$, the situation $sp/(p-1)<1$ corresponds to $\Theta>1$, and
$sp/(p-1)>1$ corresponds to $\Theta<1$. 

In the case that $\Theta<1$, Theorem~1.1 of~\cite{BT}
claims that the solution is locally Lipschitz continuous, whereas this is not possible in the more general
setting of metric measure spaces as even $2$-harmonic functions are at best guaranteed only
to be $\tau$-H\"older continuous, see for instance~\cite{KRS}.

\end{document}